\documentclass[a4paper,twoside,12pt]{amsart}

\usepackage{amssymb}
\usepackage{amsmath}
\usepackage{epsfig}
\usepackage{graphicx}
\usepackage{hyperref}
\usepackage{subfigure}

\usepackage[latin1]{inputenc} % Permite escribir e\~{n}es y tildes directamente
\usepackage[T1]{fontenc}

\newtheorem{theorem}{Theorem}[section]
\newtheorem{lemma}[theorem]{Lemma}
\newtheorem{prop}[theorem]{Proposition}
\newtheorem{cor}[theorem]{Corollary}

\theoremstyle{definition}
\newtheorem{definition}[theorem]{Definition}

\newcommand\bn{{\mathbb N}}
\newcommand\bz{{\mathbb Z}}
\newcommand\bq{{\mathbb Q}}
\newcommand\bc{{\mathbb C}}

\newcommand\ba{{\mathbb A}}

\newcommand\br{{\mathbb R}}

\newcommand\mc{{\mathcal M}}
\newcommand\dc{{\mathcal D}}
\newcommand\Ic{{\mathcal I}}

\newcommand\nc{{\mathcal N}}

\newcommand\cl{{\mathcal L}}

\newcommand\cls{{\mathcal S}}
\newcommand\Gc{{\mathcal G}}

\newcommand\rc{{\mathcal R}}

\newcommand\pc{{\mathcal P}}

\newcommand\bl{{\mathbb L}}

\newcommand\tS{{\operatorname{Supp}}}

\newcommand\ti{{\operatorname{in}}}

\newcommand\ag{{\alpha}}
\newcommand\bg{{\beta}}
\newcommand\Dg{{\Delta}}
\newcommand\mg{{\mu}}
\newcommand\sg{{\sigma}}
\newcommand\Sg{{\Sigma}}
\newcommand\lbg{{\lambda}}
\newcommand\rg{{\rho}}
\newcommand\og{{\omega}}
\newcommand\zg{{\zeta}}
\newcommand\Gg{{\gamma}}

\numberwithin{equation}{section}

\theoremstyle{remark}
\newtheorem{remark}[theorem]{Remark}

\numberwithin{equation}{section}

\title{Newton trees for ideals in two variables and applications }

\author{Pierrette Cassou-Nogu\`es}
\address{Institut de Math\'ematiques de Bordeaux, Universit\'e Bordeaux I,
350, Cours de la Lib\'eration, 33405, Talence Cedex 05, FRANCE}
\email{cassou@math.u-bordeaux1.fr}
\thanks{}

\author{Willem Veys}
\address{K.U.Leuven, Dept. Wiskunde, Celestijnenlaan 200B, 3001 Leuven, Belgium}
\curraddr{}
\email{wim.veys@wis.kuleuven.be}
\thanks{}

\thanks{
First author is partially supported  by the grants
MTM2010-21740-C02-01 and MTM2010-21740-C02-02.}

\keywords{}

\dedicatory{Dedicated to Jan Denef on the occasion of his sixtieth birthday}

\subjclass[2000]{}

\begin{document}
\begin{abstract}
We introduce an efficient way, called Newton algorithm, to study arbitrary ideals in $\bc[[x, y]]$,
using a finite succession of Newton polygons. We codify most of the data of the algorithm in a useful combinatorial object, the Newton tree.
For instance when the ideal is of finite codimension, invariants like integral closure and Hilbert-Samuel multiplicity were already combinatorially determined in the very special cases of monomial or non degenerate ideals, using the Newton polygon of the ideal.
With our approach, we can generalize these results to arbitrary ideals. In particular
the Rees valuations of the ideal will correspond
to the so-called dicritical vertices of the tree, and its Hilbert-Samuel multiplicity
has a nice and easily computable description in terms of the tree.
\end{abstract}

\maketitle

\section{Introduction}\label{sec-intro}

Let $\Ic$ be an ideal in $\bc[[x,y]]$, given by a system of
generators $\Ic=(f_1,\cdots,f_r)$.  How to compute its integral
closure $\overline{\Ic}$, its Hilbert-Samuel multiplicity $e(\Ic)$ (when $I$ is of finite codimension) and so on?

The simplest case already studied is the case of monomial ideals, which
means ideals generated by monomials. In this case,  the results are
expressed in terms of the Newton polygon of the ideal. If $f=\sum
c_{\ag,\bg}x^{\ag}y^{\bg} \in \bc[[x,y]]$, denote $\tS
f=\{(\ag,\bg)\in \bn \times \bn \mid c_{\ag,\bg}\neq 0\} .$ The
Newton polygon $\mathcal{N}(\Ic)$ is the union of the
$1$-dimensional compact faces of the convex hull $\Delta(\Ic)$ of
$\cup_i \tS f_i$. In particular the Zariski decomposition of $\overline{\Ic}$ takes the easy form
\begin{equation}\label{closure-monomial}
\overline{\Ic}=x^ky^l\prod_{S\in \mathcal{N}(\Ic)}\Ic_{(p_S,q_S)}^{\delta_S},
\end{equation}
where $S$ is a face of $\mathcal{N}(\Ic)$ with equation
$p_S\ag+q_S\bg=N_S$, with $(p_S,q_S)=1$ and $\delta_S=N_S/(p_Sq_S)$,
and $\Ic_{(p_S,q_S)}$ is the  integrally closed simple ideal such
that its Newton polygon has a unique face with equation
$p_S\ag+q_S\bg=N_S/\delta_S$. Moreover, if $\Ic$ has finite
codimension, we have
$$e(\Ic)=\sum_{S\in \mathcal{N}(\Ic)} N_S\delta_S,$$ which is twice the area of the region $\br^2_+\setminus \Delta(\Ic)$ delimited by the coordinate axes and the Newton polygon.
Such results have been generalized in the case of non degenerate
ideals of finite codimension.

In this article we prove for instance that such results can be
generalized to {\it any} ideal, if we use a finite number of Newton
polygons instead of one.
The method we use is inspired by Newton's method to find roots of
$f\in \bc[[x,y]]$. We call it Newton algorithm.  We codify the
algorithm in two ways: the Newton tree which keeps the information
of the successive Newton polygons we encounter in the algorithm, and
the Newton process which keeps the information of all the Newton
maps we use in the algorithm. From the Newton process, we can
recover the Newton tree but not the other way around. We prove that
the Newton process characterizes the integral closure of the ideal
and allows to give its Zariski decomposition, whereas the Newton
tree suffices to compute the Hilbert-Samuel multiplicity.

It should be noted that the Newton algorithm has been very efficient in the proof of the monodromy conjecture for quasiordinary power series, see \cite{ACLM1}\cite{ACLM2}.

\bigskip
The paper is organized as follows.
In the next section we introduce the Newton algorithm, as a composition of Newton maps, that transforms an arbitrary ideal $\Ic$ of
$\bc[[x,y]]$ into a principal \lq monomial-like\rq\ ideal. We associate a notion of depth to $\Ic$, being roughly the length of this algorithm,
and we say that $\Ic$ is non degenerate if the depth of $\Ic$ is at most one. In the cases where $\Ic$ is principal or of finite codimension, this is consistent with usual notions of non degeneracy. In section 3 we describe the Newton tree and its combinatorial properties. In particular we compare the Newton tree of a non-principal ideal with the Newton tree of a generic curve of that ideal. The Newton process is treated in section 4, where we show that an ideal has the same Newton process as its integral closure. In section 5 we then study several invariants of an ideal $\Ic$ of $\bc[[x,y]]$ in terms of its Newton tree or process. First we identify the Rees valuations of $\Ic$ with certain elements in the Newton tree/process. This leads to the proof of the fact that two ideals have the same Newton process if and only if they have the same integral closure, using a result of Rees and Sharp. Further we interpret Zariski's decomposition of $\Ic$, as a product of principal and  simple integrally closed ideals of finite codimension, in terms of the Newton process, generalizing (\ref{closure-monomial}). Finally we compute, when $\Ic$ has finite codimension, its multiplicity, its \L ojasiewicz exponent and its Hilbert-Samuel multiplicity $e(\Ic)$ from the Newton tree, and we show an alternative formula for $e(\Ic)$ in terms of the area defined by the successive Newton
polygons encountered in the Newton algorithm.

\section{Newton Algorithm for an ideal}

\smallskip
\subsection{Newton polygon}
For any set $E\subset \bn \times \bn$, denote by $\Delta(E)$ the smallest convex set containing
$E+\br _{+}^2=\{a+b \mid a\in E, b\in \br _{+}^2\}$. A set $\Delta \subset \br ^2$ is a { \it Newton diagram}
if there exists a set  $E\subset \bn \times \bn$ such that $\Delta=\Delta (E)$. The smallest set $E_0\subset
 \bn \times \bn$ such that $\Delta=\Delta (E_0)$ is called the set of {\em vertices} of a Newton diagram $\Delta$; it is a finite set. Let $E_0=\{v_0,\cdots, v_m\}$, with  $v_i=(\ag _i,\bg_i)\in \bn \times \bn$ for $i=0,\cdots, m$, and $\ag_{i-1}<\ag_i$, $\bg_{i-1}>\bg_i$ for $i=1,\cdots, m$. For $i\in\{1,\cdots, m\}$, denote
 $S_i=[v_{i-1},v_i]$ and by $l_{S_i}$ the line supporting the segment $S_i$. We call $\nc (\Dg)=\cup_{1\leq i\leq m} S_i$ the {\it Newton polygon} of $\Dg$ and the $S_i$ its {\it faces}. The Newton polygon $\nc (\Dg)$ is empty if and only if
 $E_0 =\{v_0=(\ag_0,\bg_0)\}$. The integer $h(\Dg)=\bg _0-\bg_m$ is called the {\it height} of $\Dg$.
 Let
 $$f(x,y)=\sum_{(\ag,\bg)\in \bn \times \bn}c_{\ag,\bg}x^{\ag}y^{\bg}\in \bc[[x,y]] .$$
 We define the support of $f$ as
 $$\tS f=\{(\ag,\bg)\in \bn \times \bn \mid c_{\ag,\bg}\neq 0\} .$$
 We denote $\Delta(f)=\Delta(\tS f)$ and $\nc (f)=\nc(\Dg (f))$.
 Let $l$ be a line in $\br^2$. We define the initial part of $f$ with respect to $l$ as
 $$\ti (f,l)=\sum_{(\ag,\bg)\in l}c_{\ag,\bg}x^{\ag}y^{\bg} .$$
 If the line $l$ has equation $p\ag+q\bg =N$, with $(p,q)\in (\bn^*)^2$ and $\gcd (p,q)=1$, then $\ti (f,l)$ is zero or a monomial or, if $l=l_S$ for some segment $S$ of $\nc (\Dg)$, of the form
 $$\ti (f,l)=x^{a_l}y^{b_l}F_S(x^q,y^p) ,$$
   where $(a_l,b_l)\in \bn^2$ and
 $$F_S(x,y)=c\prod _{1\leq i\leq n }(y-\mg _ix)^{\nu_i} ,$$
 with $c\in \bc^*$, $n\in \bn^*$, $\mg_i\in \bc^*$ (all different) and $\nu_i \in \bn^*$.

 %\smallskip
 For example if $f(x,y)=(x^2+y^3)^2+\sum_{i\geq1} x^iy^{5i}$, then $\nc(f)=\{S\}$, where $l_S$ has equation $3\alpha+2\beta=12$ and $\ti (f,l_S)=(x^2+y^3)^2$.

 \bigskip
 Now let  $\Ic=(f_1,\cdots, f_r)$ be a non-trivial ideal in $\bc[[x,y]]$. We define
 $$\Dg(\Ic)=\Dg(\cup _{1\leq i\leq r}\tS f_i)  \quad\text{and}\quad \nc(\Ic)=\nc(\Dg(\Ic)) .$$
 When $\Ic=(f)$ we simply write $\Dg(f)$ and $\nc(f)$.
For a segment $S$ of $\nc(\Ic)$ we denote by $\ti (\Ic,S)$ the ideal generated by the
 $\ti (f_i,l_S), 1\leq i\leq r$, and call it the {\it initial ideal} of $\Ic$ with respect to $S$.

\begin{lemma}
 The sets $\Dg(\Ic)$ and $\nc(\Ic)$ and the ideals  $\ti (\Ic,S)$ depend only on $\Ic$, not on a system of generators of $\Ic$.
 \end{lemma}

 \begin{proof}
 If $\Ic=(f_1,\cdots,f_m)=(g_1,\cdots ,g_s)$, then we have for $j=1,\cdots ,s$ that
 $$g_j(x,y)=\sum_ih_{ij}(x,y)f_i(x,y) ,$$
 where the $h_{ij} \in \bc[[x,y]]$. Then
 $$\tS \, g_j \subset \cup _i \tS f_i +\br_+^2$$
 for all $j$, and hence
 $$\cup _j \tS \, g_j +\br_+^2 \subset \cup _i \tS f_i +\br_+^2 .$$
 By symmetry we can conclude that indeed
 $$\cup _j \tS \, g_j +\br_+^2 = \cup _i \tS f_i +\br_+^2 .$$

 Moreover, since any such segment $S$ is a face of  $\nc(\Ic)$,
 we have for $j=1,\cdots, s$ that
 $$\ti (g_j,l_S)=\sum_i h_{ij}(0,0) \, \ti(f_i,l_S) ,$$
 and we can conclude analogously.
 \end{proof}

\begin{remark}
The proof of the lemma shows that $\Dg(\Ic)=\Dg(\cup _{f\in \Ic}\tS f)$.
\end{remark}

 \begin{remark}
 The Newton polygon of an ideal is empty if and only if the ideal is principal, generated by a monomial.
 \end{remark}

 %Let $\Ic=(f_1,\cdots, f_r)$ and $S\in \nc(\Ic)$. We denote by $\ti (\Ic,S)$ the ideal generated by
 %$\ti (f_i,l_S)$ for $i=1,\cdots, r$. This ideal doesn't depend on the chosen generators. Indeed, assume
 %$\Ic=(f_1,\cdots,f_m)=(g_1,\cdots ,g_s)$, then for $j=1,\cdots ,s$,
 %$$g_j(x,y)=\sum_ih_{ij}(x,y)f_i(x,y)$$
 %Since $S$ is a face of $\nc(\Ic)$, we have for $j=1,\cdots, s$
 %$$\ti (g_j,l_S)=\sum_i h_{ij}(0,0)\ti(f_i,l_S)$$

\medskip
 Let $S$ be a face of $\nc(\Ic)$ and $p_S\ag+q_S\bg =N_S$ be the equation of $l_S$, with $\gcd (p_S,q_S)=1$ as before. Then $\ti(\Ic,S)$ is of the form
 \begin{equation}\label{initialideal1}
 \ti(\Ic,S)=\left(x^{a_S}y^{b_S}F_{\Ic,S}(x^{q_S},y^{p_S})\right)
 \end{equation}
 or
 \begin{equation}\label{initialideal2}
 \ti(\Ic,S)=x^{a_S}y^{b_S}F_{\Ic,S}(x^{q_S},y^{p_S})\big(k_1(x^{q_S},y^{p_S}),\cdots,k_s(x^{q_S},y^{p_S})\big)
 \end{equation}
 with $s \geq 2$,
 where $F_{\Ic,S},k_1,\cdots,k_s$ are homogeneous polynomials, $F_{\Ic,S}$ is not divisible by $x$ or $y$  and $k_1,\cdots, k_s$ are coprime and of the same degree $d_S$. In the first case we put $d_S=0$.
 The polynomial $F_{\Ic,S}$, monic in $y$,  is called the {\it face polynomial} (it can be identically one). A face $S$ is called a {\it dicritical face} if $\ti(\Ic,S)$ is not a principal ideal.  Thus it is dicritical if and only if $d_S\geq 1$.

 The following equality is an immediate consequence of the definitions.

\begin{lemma}
Let $\Ic$ be a non-trivial ideal in $\bc[[x,y]]$. Write the face polynomial of each face $S$ of its Newton polygon  $\nc(\Ic)$ in the form
$$
F_{\Ic,S}(x,y)=\prod_{1\leq i\leq n_S }(y-\mg_{S,i}x)^{\nu_{S,i}}
,$$ where $n_S\in \bn$, $\mg_{S,i}\in \bc^*$ (all different) and
$\nu_{S,i} \in \bn^*$. Then the height $h$ of  $\Dg(\Ic)$ satisfies

\begin{equation}\label{hauteur0}
h=\sum _S p_S \left( d_S+\sum_{1\leq i\leq n_S }\nu_{S,i}\right) .
\end{equation}
\end{lemma}

\medskip

\medskip
 \subsection{Newton maps}

 \begin{definition}
 Let $(p,q)\in (\bn^*)^2$ with $\gcd (p,q)=1$. Take $(p',q')\in \bn^2$ such that $pp'-qq'=1$. Let $\mu\in \bc^*$. Define
$$ \begin{matrix}
 &\sg_{(p,q,\mu)}&:&\bc[[x,y]]&\longrightarrow&\bc[[x_1,y_1]]\\
 &&&f(x,y)&\mapsto&f(\mu^{q'}x_1^p,x_1^q(y_1+\mu^{p'})) .
 \end{matrix}$$
 We say that the map $\sg_{(p,q,\mu)}$ is a {\it Newton map}.
  \end{definition}

\begin{remark} (1) The numbers $(p',q')$ are introduced only to avoid taking roots of complex numbers.

(2) Let $(p',q')$ be such that $pp'-qq'=1$. For $i\in \bn$, we have
  $$f\left(\mu^{q'+ip}x_1^p,x_1^q(y_1+\mu^{p'+iq})\right)=f\left(\mu^{q'}(\mu^i x_1)^p,(\mu^i x_1)^q((\mu^{-iq}y_1)+\mu^{p'})\right) ,$$
  which shows that the change of $(p',q')$ into $(p'+iq,q'+ip)$ corresponds to the change of coordinates
  $(x_1,y_1)\mapsto (\mu ^i x_1,\mu^{-iq} y_1)$.
\end{remark}

In the sequel we will always assume that $p'\leq q$ and $q'<p$. This will make procedures canonical.

  \begin{lemma}\label{Newton-alg}
  Let $f(x,y)\in \bc[[x,y]]$, $f\neq 0$ and $\sg_{(p,q,\mu)}(f)(x_1,y_1)=f_1(x_1,y_1)\in \bc[[x_1,y_1]]$.
  \begin{enumerate}
  \item If there does not exist a face $S$ of $\nc(f)$ whose supporting line has equation $p\ag +q\bg=k$ with $k\in \bn$, then
  $$f_1(x_1,y_1)=x_1^m u(x_1,y_1) $$
  with $m\in \bn$, $u(x_1,y_1)\in \bc[[x_1,y_1]]$ and $u(0,0)\neq 0$.
  \item If there exists a face  $S$ of $\nc(f)$ whose supporting line has equation $p\ag +q\bg =k_0$ for
  some $k_0 \in \bn$,  and if $F_{S}(1,\mu)\neq 0$, then
  $$f_1(x_1,y_1)=x_1^{k_0} u(x_1,y_1) $$
  with $ u(x_1,y_1)\in \bc[[x_1,y_1]]$ and $u(0,0)\neq 0$.
  \item If there exists a face $S$ of $\nc(f)$ whose supporting line has equation $p\ag +q\bg =k_0$ for
  some $k_0 \in \bn$,  and if $F_{S}(1,\mu)= 0$, then
  $$f_1(x_1,y_1)=x_1^{k_0} g_1(x_1,y_1)$$
  with $ g_1(x_1,y_1)\in \bc[[x_1,y_1]]$ and $g_1(0,0)= 0,g_1(0,y_1)\neq 0$.
  \end{enumerate}
  \end{lemma}

  \begin{proof}
  Let
  $$f(x,y)=\sum_{(\ag,\bg)\in \bn^2} c_{\ag,\bg}x^{\ag}y^{\bg}\in \bc[[x,y]] ;$$
we write
 $$f(x,y)=\sum_{k\in \bn}\sum_{p\ag+q\bg=k}c_{\ag,\bg}x^{\ag}y^{\bg} .$$
 We have for all $k\in \bn$ that
 $$\sum_{p\ag+q\bg=k}c_{\ag,\bg}(\mu^{q'}x_1^p)^{\ag}(x_1^q(y_1+\mu^{p'}))^{\bg}=x_1^k\sum_{p\ag+q\bg=k}c_{\ag,\bg}\mu^{q'\ag}(y_1+\mu^{p'})^{\bg} .$$
 If $k_0$ is the smallest $k$ such that
 $\sum_{p\ag+q\bg=k}c_{\ag,\bg}x^{\ag}y^{\bg}\neq 0$ (that is, such that at least one $c_{\ag,\bg}$ is nonzero),
 then
$$ f_1(x_1,y_1)=x_1^{k_0}\left(\sum_{k\geq k_0}x_1^{k-k_0}\sum_{p\ag+q\bg=k}c_{\ag,\bg}\mu^{q'\ag}(y_1+\mu^{p'})^{\bg}\right) .$$

If there does not exist a face of $\nc(f)$ whose supporting line has equation $ p\ag+ q\bg=k_0$, then
for $k=k_0$ there is only one nonzero $c_{\ag,\bg}$ in the sum above, hence
$$f_1(x_1,y_1) = c_{\ag,\bg}x_1^{k_0}\mu^{q'\ag}(y_1+\mu^{p'})^{\bg} + x_1^{k_0+1} (\cdots)$$
and, since $\mu\neq 0$, we see that $f_1(x_1,y_1)=x_1^{k_0}u(x_1,y_1)$ with $u(0,0)\neq 0$.

If there exists a face $S$ of $\nc(f)$ whose supporting line has equation $ p\ag+ q\bg=k_0$, we write $\ti (f,l_S)=x^{a}y^{b}F_S(x^q,y^p)$ as above and $F_S(x,y)= c \prod_i (y-\mu_i x)^{\nu_i}$.   Then $f_1(x_1,y_1)$ is of the form
$$f_1(x_1,y_1)=x_1^{k_0}\left( c\mu^{q'a}(y_1+\mu^{p'})^b F_{S}(\mu^{qq'},(y_1+\mu^{p'})^p)+ \sum_{k> k_0}x_1^{k-k_0}\sum_{p\ag+q\bg=k}c_{\ag,\bg}\mu^{q'\ag}(y_1+\mu^{p'})^{\bg}) \right),$$
where
$$F_{S}(\mu^{qq'},(y_1+\mu^{p'})^p)=\prod_i((y_1+\mu^{p'})^p-\mu_{i}\mu^{qq'})^{\nu_{i}}=
F_{S}(1,\mu)\mu^{qq'\sum_i\nu_{i}}+y_1G_S(y_1)$$
with $G_S(y_1)\in \bc[y_1]$ and $G_S(y_1)\neq 0$.
So, if $F_{S}(1,\mu)\neq 0$, then
$f_1(x_1,y_1)=x_1^{k_0}u(x_1,y_1)$ with $u(0,0)\neq 0$, and if $F_{S}(1,\mu)= 0$, then
$f_1(x_1,y_1)=x_1^{k_0}g_1(x_1,y_1)$ with $g_1(0,0)= 0$.

Note that $F_{S}(1,\mu)= 0$ if and only if $\mu=\mu_{i}$ for some $i$ and $g_1(0,y_1)=y_1^{\nu_{i}}+\cdots $, where $\cdots$ means higher degree terms in $y_1$.
\end{proof}

\begin{remark}
In the first and second case of Lemma \ref{Newton-alg}, the Newton polygon of $f_1$ is empty. In the third case,  the height of the
Newton diagram of $f_1$ is less than or equal to the multiplicity of $\mu$ as root of $F_S(1,X)$.
\end{remark}

Let $\Ic=(f_1,\cdots, f_r)$ be a non-trivial ideal in $\bc[[x,y]]$. Let $\sg_{(p,q,\mu)}$ be a Newton map. We denote by
$\sg_{(p,q,\mu)}(\Ic)$ the ideal in $\bc[[x_1,y_1]]$  generated by the $\sg_{(p,q,\mu)}(f_i)$ for $i=1,\cdots, r$. Since a Newton map is a ring homomorphism, this ideal does not depend on the choice of the generators of $\Ic$.

\begin{lemma}\label{Newton-algbis}
Let $\Ic$ be a non-trivial ideal in $ \bc[[x,y]]$ and
$\sg_{(p,q,\mu)}(\Ic)=\Ic_1$.
\begin{enumerate}
\item If there does not exist a face $S$ of $\nc(\Ic)$ whose supporting line has equation $p\ag +q\bg=N$ with $N\in \bn$,  then the ideal $\Ic_1$ is principal,
generated by a power of $x_1$.
\item If there exists a face  $S$ of $\nc(\Ic)$ whose supporting line has equation $p\ag +q\bg =N_0$ for
  some $N_0 \in \bn$,  and if $F_{\Ic, S}(1,\mu)\neq 0$, then $\Ic_1=(x_1^{N_0})$.
\item If there exists a face $S$ of $\nc(\Ic)$ whose supporting line has equation $p\ag +q\bg =N_0$ for
  some $N_0 \in \bn$,  and if $F_{\Ic, S}(1,\mu)= 0$, then $\Ic_1=(x_1^{N_0})\Ic _1'$ and the height of the Newton polygon of $\Ic_1$ is less than or equal to the multiplicity of $\mu$ as root of $F_{\Ic, S}(1,X)$,
    \end{enumerate}
  \end{lemma}

\begin{proof}
  The two first assertions are consequences of the previous lemma. We prove the third one.
  Let $S$ be the face of the Newton polygon of $\Ic$ with equation $p\ag+q\bg=N_0$.
  We denote
  $$ \ti(\Ic,S)=x^{a_S}y^{b_S}F_{\Ic,S}(x^{q},y^{p})\left(k_1(x^{q},y^{p}),\cdots,k_s(x^{q},y^{p})\right)$$
  as in (\ref{initialideal1}) or (\ref{initialideal2}) with $\deg k_i=d_S\geq 0$.
  We consider $\sg_{(p,q,\mu)}$ with $\mu$ a root of $F_{\Ic,S}$ of multiplicity $\nu$.
  We may assume that
  $$\ti (f_i,S)=x^{a_S}y^{b_S}F_{\Ic,S}(x^{q},y^{p})k_i(x^q,y^p)$$
  with $\deg k_i=d_S$ if $i=1,\cdots, s$ and $k_i=0$ if $i>s$.

  If $i\in \{1,\cdots,s\}$ we have
  $$f_{i,1}(x_1,y_1)=x_1^{N_0}g_i(x_1,y_1),\ \text{where}\ g_i(0,y_1)=y_1^{\nu_i}+...$$
  and $\nu_i$ is the multiplicity of $\mu$ in $F_{\Ic,S}(1,X)k_i(1,X)$. Hence $\nu_i\geq \nu$. Since the greatest common divisor of the $k_i$ is one, there exists $i\in \{1,\cdots,s\}$ such that $\nu_i=\nu$.

  If $i>s$, we consider a line parallel to $l_S$ which hits the Newton polygon of $f_i$, and
 $$f_{i,1}(x_1,y_1)=x_1^{n_i}g_i(x_1,y_1) $$
 with $n_i>N_0$. We conclude that the Newton polygon of $\Ic _1$ has height less than or equal to $\nu$.
  \end{proof}

\medskip
 \subsection{Newton algorithm}

%Assume that the ideal $\Ic$ is in good coordinates.
%\bigskip
Given an ideal $\Ic$ in $\bc[[x,y]]$ and a Newton map $\sg_{(p,q,\mu)}$, we denote by  $\Ic_{\sg}$ the ideal $\sg_{(p,q,\mu)}(\Ic)$. Consider a sequence $\Sg_n=(\sg_1,\cdots, \sg_n)$  of length $n$ of Newton maps. We define $\Ic _{\Sg_n}$ by induction:
$$\Ic_{\Sg_1}=\Ic_{\sg_1},\cdots, \Ic_{\Sg_i}=(\Ic_{\Sg_{i-1}})_{\sg_i},\cdots,\Ic_{\Sg_n}=(\Ic_{\Sg_{n-1}})_{\sg_n}  .$$

\medskip\begin{theorem}
Let $\Ic$ be a non-trivial ideal in $\bc[[x,y]]$.  There exists an
integer $n_0$ such that, for any sequence $\Sg_n=(\sg_1,\cdots,
\sg_n)$ of Newton maps of length at least $ n_0$, the ideal $\Ic
_{\Sg_n}$ is principal, generated by $x^k(y+h(x))^{\nu}$ with $h\in
x\bc [[x]]$ and $(k,\nu)\in \bn\times \bn$.
\end{theorem}

\begin{proof}
If $\mathcal{N}(\Ic) $ is empty, then $\Ic$ is principal and generated by a monomial.  We can take $n_0=0$.

Assume that $\mathcal{N}(\Ic) $ is not empty and has height $h>0$.
For all the Newton maps $\sg_{(p,q,\mu)}$ such that $p\ag+q\bg=N$ is
not an equation of a face of the Newton polygon or $\mu $ is not a
root of the face polynomial, the ideal $\Ic_{\sg}$ is principal,
generated by a monomial. If there is a face whose supporting line
has equation $p\ag+q\bg=N$ and $\mu$ is a root of its face
polynomial, then the Newton polygon of $\Ic_{\sg}$ has height  less
than or equal to $h$. Then either we end with a principal ideal
generated by a monomial or the heights of the Newton polygons
stabilize to a constant positive value. We study the case where the height
remains constant in the following lemma's, what will finish the
proof.
\end{proof}

The first lemma is straightforward.

\begin{lemma}\label{constant0}

The height of  $\Ic_{\sg}$ is equal to the height of $\Ic$ if and only if the Newton polygon of $\Ic$ has a unique face $S$ with
$$\ti (\Ic,S)=\left( x^ky^l (y-\mu x^q)^{\nu} \right)$$
and $\sg=\sg_{(1,q,\mu)}$.
\end{lemma}

When the height of the Newton polygon stabilizes in the Newton algorithm, we say that the Newton algorithm  stabilizes.

\begin{lemma}\label{constant1}

Let $\Ic$ be a principal ideal whose Newton algorithm stabilizes. Then
$$\Ic =(x^ky^l(y+h(x))^{\nu})$$
with $h(x)\in x\bc [[x]]$.

\end{lemma}

\begin{proof}
It is a consequence of Newton's method to find roots of $f(x,y)\in \bc[[x,y]]$. See for example Theorem 2.1.1. in Wall's book \cite{Wall}.
\end{proof}

\begin{lemma}\label{constant2}
Let $\Ic=(f_1,\cdots, f_r)$ be an ideal whose Newton algorithm stabilizes. Then
$$\Ic=(f)\Ic_1,$$
where the Newton algorithm of $f$ stabilizes and $\mathcal{N}(f)$ has height at least $1$.
\end{lemma}

\begin{proof}
We may assume that $\mathcal{N}(\Ic)$ has a unique face $S$ and
$$\ti(\Ic,S)= \left( x^ky^l(y-\mu x^q)^{\nu} \right).$$
Then there exists $i\in \{1,\cdots,r\}$ such that $\ti(f_i,S)=c_ix^ky^l(y-\mu x^q)^{\nu}$ with $c_i\in \bc^*$.
Let
$$I=\{i\in \{1,\cdots,r\}\mid \ti(f_i,S)=c_ix^ky^l(y-\mu x^q)^{\nu} \ \text{for some}\ c_i\in \bc^* \}$$
If $i\notin I$ then $\ti(f_i,S)=0$.

Let $\sg=\sg_{(1,q,\mu)}$.
For $i\in I$, we have $f_i(x_1,y_1)=x_1^{k+ql+q\nu}g_i(x_1,y_1)$ with $g_i(0,y_1)=y_1^{\nu}+\cdots$. Then, as the Newton algorithm of $\Ic$ stabilizes, for all $i\in I$ the Newton algorithm of $f_i$ stabilizes and
$$f_i(x,y)=x^ky^l(y+h(x))^{\nu}u_i(x,y),$$ where $u_i(x,y)$ is a unit in $\bc[[x,y]]$.

For $i\notin I$, consider the parallel line $l_i$ to $S$ which hits the Newton polygon of $f_i$.  The initial part  $\ti (f_i, l_i)=x^{a_i}y^{b_i}F_i(x^q,y)$, where $F_i$ can be a constant. If $F_i$ is a constant or is not divisible by $(y-\mu x^q)$, then
$$f_i(x_1,y_1)=x^{n_i}u_i(x_1,y_1)$$ where $u_i$ is a unit and $n_i > k+ql+q\nu$. But if there exists $i\notin I$ such that $f_i(x_1,y_1)=x^{n_i}u_i(x_1,y_1)$ with $n_i > k+ql+q\nu$, then after a finite number of steps we have a dicritical face and the algorithm does not stabilize further. We conclude that $F_i(x^q,y)$ is divisible by $(y-\mu x^q)$ for all $i\notin I$. Consequently $\sg$ belongs to the Newton algorithm of $f_i$ for all $i\in \{1,\cdots,r\}$.

We showed that the Newton algorithm of $y+h(x)$ appears in the Newton algorithm of $f_i$ for all $i\in \{1,\cdots,r\}$.  Therefore the $f_i$ have a common factor with height at least one with a single root $y+h(x)$, and indeed $\Ic$ is of the form $\Ic=(f)\Ic_1$ as stated.
\end{proof}

\begin{lemma}\label{constant3}

If the Newton algorithm of $\Ic$ stabilizes, then
$$\Ic=\left( x^ky^l(y+h(x))^{\nu} \right)$$
\end{lemma}

\begin{proof}
We use induction on the height and the previous lemma.
\end{proof}

\bigskip
\noindent
{\bf Example 1.}
We consider in $\bc[[x,y]]$ the ideal
$$\Ic=\left(y^4(y+x)(y^2-3x),((y+x)^3+x^8)(y^2-3x)\right).$$
Its Newton polygon is given in Figure 1.

  \begin{figure}[ht]
 \begin{center}
\includegraphics{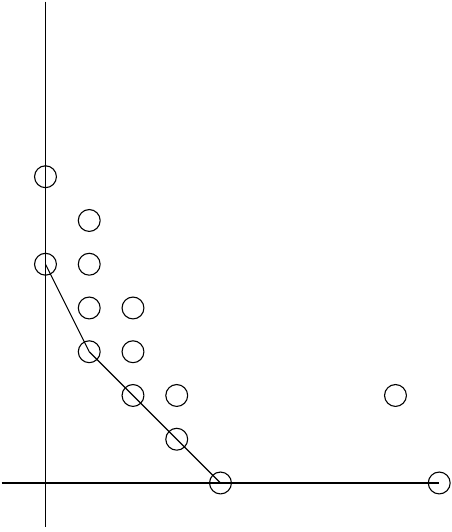}
\caption{}
 \end{center}
 \end{figure}

The faces $S_1$ and $S_2$ have supporting lines with equations $2\alpha+\beta=5$ and $\alpha+\beta=4$, respectively. The initial ideals of $\Ic$ with respect to these segments are
$$\ti (\Ic,S_1) = \left(y^3(y^2-3x)\right) \qquad\text{and}\qquad \ti (\Ic,S_2)= \left(x(y+x)^3\right).$$
Both segments are not dicritical; their face polynomials are $F_{\Ic,S_1}=y^2-3x$ and $F_{\Ic,S_2}=(y+x)^3$, respectively.
We first consider the Newton map $\sg_{(p,q,\mu)}=\sg_{(2,1,3)}$ associated to $S_1$ and $\mu=3$. It is given by the substitution
$$x=3x_1^2, \qquad\qquad y=x_1(y_1+3).$$
The image ideal $\Ic_1$ is given by
$$
\begin{aligned}
\Ic_1 &= \left(x_1^4 (y_1+3)^4(x_1y_1+3x_1+3x_1^2)(x_1^2(y_1+3)^2-9x_1^2)\right. ,\\
 &\qquad\qquad\qquad\qquad\qquad\qquad\left. ((x_1y_1+3x_1+3x_1^2)^3+3^8x_1^{16})(x_1^2(y_1+3)^2-9x_1^2)\right)\\
&=\left(x_1^7(y_1^2+6y_1),x_1^5(y_1^2+6y_1)\right)\\
&= (x_1^5y_1).
\end{aligned}
$$
It is a monomial ideal, hence we stop the procedure for $S_1$.

Next we consider the Newton map $\sg_{(p,q,\mu)}=\sg_{(1,1,-1)}$ associated to $S_2$ and $\mu=-1$. It is given by the substitution
$$x=x_1, \qquad\qquad y=x_1(y_1-1).$$
The image ideal $\Ic_1$ is given by
$$
\begin{aligned}
\Ic_1 &= \left(x_1^4(y_1-1)^4 x_1y_1(x_1^2(y_1-1)^2-3x_1), (x_1^3y_1^3+x_1^8)(x_1^2(y_1-1)^2-3x_1)\right)\\
&=\left(x_1^6y_1,x_1^4(y_1^3+x_1^5)\right)\\
&= x_1^4(x_1^2y_1,y_1^3+x_1^5).
\end{aligned}
$$
Its Newton polygon is given in Figure 2.

  \begin{figure}[ht]
 \begin{center}
\includegraphics{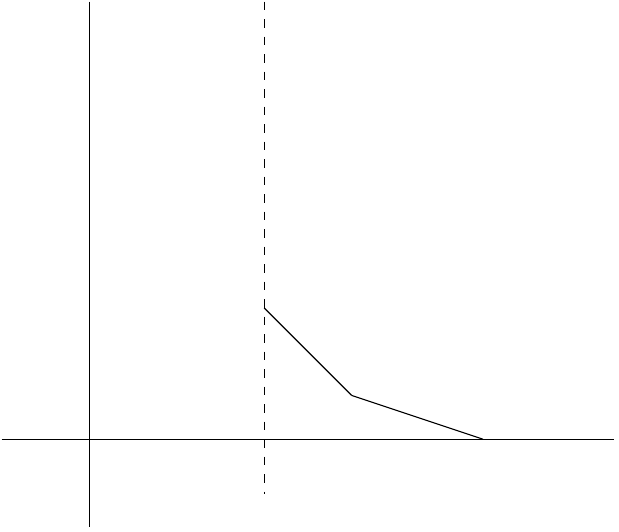}
\caption{}
 \end{center}
 \end{figure}

The faces $S'_1$ and $S'_2$ have supporting lines with equations $\alpha+\beta=7$ and $\alpha+3\beta=9$, respectively. The initial ideals of $\Ic_1$ with respect to these segments are
$$\ti (\Ic_1,S'_1) = x_1^4y_1(x_1^2,y_1^2) \qquad\text{and}\qquad \ti (\Ic_1,S'_2)= x_1^6(y_1,x_1^3).$$
Both segments are dicritical and have constant face polynomial; their degrees are $d_{S'_1}=2$ and $d_{S'_2}=1$.

We continue with the Newton map $\sg_{(p,q,\mu)}=\sg_{(1,1,\mu)}$ associated to $S'_1$ and $\mu\in \bc^*$. It is given by the substitution
$$x_1=x_2, \qquad\qquad y_1=x_2(y_2+\mu).$$
The image ideal $\Ic_2$ is given by
$$
\Ic_2 = x_2^4 \left(x_2^3, x_2^3(y_2+\mu)^3+x_2^5\right) = (x_2^7).
$$
Analogously we consider the Newton map $\sg_{(p,q,\mu)}=\sg_{(1,3,\mu)}$ associated to $S'_2$ and $\mu\in \bc^*$. It is given by the substitution
$$x_1=x_2, \qquad\qquad y_1=x_2^3(y_2+\mu).$$
The image ideal $\Ic_2$ is given by
$$
\Ic_2 = x_2^4 \left(x_2^5, x_2^9(y_2+\mu)^3+x_2^5\right) = (x_2^9).
$$
Both ideals are monomial, hence we stop also the procedure for $S_2$.

\bigskip
It should be clear that, when a face of a Newton polygon is dicritical with constant face polynomial, the associated Newton map $\sg_{(p,q,\mu)}$ induces a monomial ideal for any $\mu\in \bc^*$. More generally, when a face $S$ of a Newton polygon is dicritical, the associated Newton map $\sg_{(p,q,\mu)}$ induces a monomial ideal for all $\mu\in \bc^*$ that are not roots of the face polynomial $F_{\Ic,S}$.  We performed explicitly the last two Newton maps in the previous example to illustrate this fact. In later examples such Newton maps will not be computed anymore.

\medskip
\begin{definition} Let $\Ic$ be a non-trivial ideal in $\bc[[x,y]]$.
We define the {\it depth} of $\Ic$, denoted by $d(\Ic)$, by
induction. If $\Ic$ is principal, generated by $x^k(y+h(x))^{\nu}$
with $h\in x\bc[[x]]$ and $(k,\nu)\in \bn\times \bn$, we say that
its depth is $0$. Otherwise, we define
$$d(\Ic)=\max d(\Ic_{\sg})+1 ,$$
where the maximum is taken over all possible Newton maps.
\end{definition}

\medskip
\subsection{Non degenerate ideals}

\begin{definition}
Let $\Ic$ be a non-trivial ideal in $\bc[[x,y]]$. We say that $\Ic$
is  {\it non  degenerate} if it is of depth at most $1$.
%(Otherwise we call $\Ic$ {\it degenerate}.)
\end{definition}

One easily verifies that monomials ideals are  non degenerate. But there are many other ones, for instance Example 2 below (taken from \cite{VV}).
On the other hand, the ideal in Example 1 above has depth $2$ and hence it is  degenerate.

\bigskip
\noindent
{\bf Example 2.}
We consider in $\bc[[x,y]]$ the ideal
$$\Ic=(x^3y,x^6+y^4).$$ Its Newton polygon is given in Figure 3.

  \begin{figure}[ht]
 \begin{center}
\includegraphics{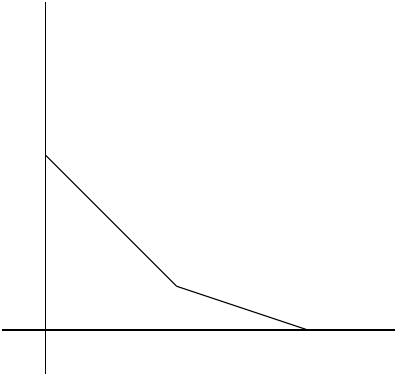}
\caption{}
 \end{center}
 \end{figure}

It has two faces $S_1$ and $S_2$ whose supporting lines have equation $\alpha+\beta=4$ and $\alpha+3\beta=6$, respectively.
The initial ideals of $\Ic$ with respect to these segments are
$$\ti (\Ic,S_1) = y(x^3,y^3) \qquad\text{and}\qquad \ti (\Ic,S_2)= x^3(y,x^3).$$
Both faces are dicritical (of degree $d_{S_1}=3$ and $d_{S_2}=1$) and have constant face polynomials. Hence all possible Newton maps will turn
$\Ic$ into a monomial ideal, and consequently its depth is $1$.

\bigskip
We now show that for the \lq extremal\rq\ cases, principal ideals or ideals of finite codimension, our notion of non degeneracy corresponds to familiar notions.

\begin{prop}
Let $\Ic=(f)$ be a principal ideal in $\bc[[x,y]]$. Then $\Ic$ is non
degenerate if and only if  the reduced curve associated to $f$ is
non degenerate (in the usual sense).
\end{prop}

\begin{proof}
Denote $g:=f_{\operatorname{red}}$ and assume that the curve
associated to $g$ is non degenerate. That means that  for all faces
$S$ of $\nc (g)$, if $(y^p-\mu x^q)^{\nu}$ (with $\mu\in \bc^*$)
divides $\ti(g,{l_S})$, then $\nu=1$.  Then after the corresponding
Newton map $\sigma_{(p,q,\mu)}$, the height of the Newton polygon is
$0$ or $1$ and the ideal $(g_\sigma)$ is generated by an element of
the form $x_1^{k}(y_1+h(x_1))$.

Reciprocally, if the depth of $\Ic =(f)$ is $0$ or $1$, after any Newton map $\sg$ we obtain $f_{\sg}(x_1,y_1)=x_1^N(y_1+h(x_1))^{\nu}u_1(x_1,y_1)$ with $u_1(0,0)\neq 0$, implying that $g_{\sg}(x_1,y_1)=x_1^{N'}(y_1+h(x_1))u_1(x_1,y_1)$. Consequently, the power of any $(y^p-\mu x^q)$ (with $\mu\in \bc^*$)  which divides $\ti(g,{l_S})$
is at most one, yielding that $g$ is non degenerate.
\end{proof}

In the sequel we often treat  non-trivial ideals in $\bc[[x,y]]$ of finite codimension. This property is equivalent to being $(x,y)$-primary and equivalent to having support $\{0\}$.

\begin{prop}
Let $\Ic=(f_1,\cdots, f_r)$ be an ideal in $\bc[[x,y]]$ of finite
codimension. Then $\Ic$ is non degenerate if and only if, for all
faces $S$ of $\nc (\Ic)$, there is no $\mu\in\bc^*$ such that
$(y^p-\mu x^q)$ divides  $ \ti (f_i,l_S)$  for all $i=1,\cdots, r$.
\end{prop}

\begin{proof}
The condition of the proposition means that, for all faces $S$ of
$\nc (\Ic)$, the face polynomial $F_{\Ic,S}=1$. This is equivalent
to the fact that the depth is $1$. (Note that if there does exist a
$\mu$ as above, then after applying the Newton map
$\sigma_{p,q,\mu}$ we do not obtain a monomial ideal.)
\end{proof}

Note that the condition in the above proposition does not depend on the choice of the generators of the ideal.
We have already seen that, if
$\Ic =(f_1,\cdots, f_r)=(g_1,\cdots, g_s)$, then we have for all $j=1,\cdots, s$ that $\ti (g_j, l_S)$ is a $\bc$-linear combination of the $\ti (f_i,l_S)$.
Then, if $y^p-\mu x^q$ divides $\ti (f_i,l_S)$ for all $i$, it divides $\ti (g_j,l_S)$ for
$j=1,\cdots, s$ and reciprocally.

%for all faces $S$ of $ \nc (\Ic)$,
%$$\operatorname{rank}\left(\begin{matrix}  \frac{\partial \ti (f_1,l_S)}{\partial x}&\cdots&\frac{\partial \ti (f_r,l_S)}{\partial x} \\
% \frac{\partial \ti (f_1,l_S)}{\partial y}&\cdots&\frac{\partial \ti (f_r,l_S)}{\partial y} \end{matrix}\right)(x,y)=2$$
%for all $(x,y) \in (\bc^*)^2$ such that $\ti (f_i,l_S)(x,y)=0$ for
%all $i=1,\cdots, r$.

\begin{remark}
(1) For an ideal $\Ic=(f_1,\cdots, f_r)$ as in the previous
proposition, being non degenerate corresponds to the (local) notion
{\it strongly non degenerate}  for the mapping $(f_1,\cdots, f_r)$,
as introduced in \cite{VZ}.

(2) In \cite{BFS} the notion of Newton non degeneracy for ideals
in $\bc[[x_1,\cdots,x_n]]$ of finite codimension is defined. For two
variables it coincides with our definition.
\end{remark}

\begin{prop}
Write the non-trivial ideal $\Ic$ in $\bc[[x,y]]$ as $\Ic=(f) \Ic'$, where the ideal $\Ic'$ is of finite codimension. Then the ideal $\Ic$ is  non degenerate if and only if the ideal $(f)$ and the ideal $\Ic'$ are both  non degenerate.
\end{prop}

\begin{proof}
The ideal $\Ic$ being  non degenerate means that, for all $(p,q,\mu)$,
$$\Ic_{\sg_{(p,q,\mu) }}=\big(x_1^N(y_1+h(x_1))^m\big).$$
This is equivalent to both
$$f_{\sg_{(p,q,\mu)}}(x_1,y_1)=x_1^{M}(y_1+h(x_1))^mf_1(x_1,y_1)$$
with $f_1(0,0)\neq 0$, and
$$\Ic'_{\sg_{(p,q,\mu)}}=(x_1^{M'}).$$
And these statements for all $(p,q,\mu)$ mean that  $(f) $ and $\Ic'$ are both  non degenerate.
\end{proof}

\medskip
  \section{Newton tree associated with an ideal}

  We collect the information of the Newton algorithm of an ideal in two different ways. The first one is the Newton tree. It keeps the tree shape of the algorithm, and the information on the successive Newton polygons. The second one, called the Newton process, will be treated in the next section. It keeps all the information of the Newton maps; the tree shape is not given explicitly but can be recovered from the data.

 \smallskip
  \subsection{Graph associated with a Newton diagram.}

  A graph associated with a Newton diagram is a vertical linear graph with vertices, edges connecting vertices and two arrows at the top and the bottom.

  If the Newton polygon is empty, that is,
  $\Dg =(N,M)+\br_+^2$,
  the graph is in Figure 4. It has one edge connecting two arrows decorated by $N$ and $M$ at the top and the bottom, respectively.

   \begin{figure}[ht]
 \begin{center}
\includegraphics{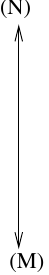}
 \caption{}
 \end{center}
\end{figure}

If the Newton polygon is $\cup_{1\leq i\leq m}S_i$, the graph has $m$ vertices $v_1,\cdots, v_m$ representing the faces $S_1, \cdots, S_m$. They are connected by edges when the faces intersect.
We add one edge at $v_1$ and at $v_m$ ended by an arrow.

  \begin{figure}[ht]
 \begin{center}
\includegraphics{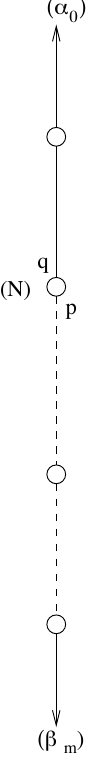}
 \caption{}
 \end{center}
\end{figure}

We decorate the vertices and the extremities of the edges near the vertices using the following rule.
Let $v$ be a vertex and $S$ be the corresponding face whose supporting line has equation $p\ag+q\bg=N$, where $(p,q)\in (\bn^*)^2$ and $\gcd (p,q)=1$.  We decorate the vertex by $(N)$. Further we decorate the extremity of the edge above the vertex with $q$ and the extremity of the edge under the vertex by $p$; we say that the decorations near $v$ are $(q,p)$. The arrows represent the  non-compact
faces with supporting lines $\{x=\ag_0\}$ and $\{y=\bg_m\}$; they are decorated with $(\ag_0)$ at the top and $(\bg_m)$ at the bottom.

%\medskip
\subsection{Newton tree of an ideal}

 We build the Newton tree of $\Ic$ by induction on the depth.
 If the depth is zero, the ideal is generated by a \lq monomial\rq\ $x^k(y+h(x))^{\nu}$; we define its Newton tree to be the graph as in Figure 6.

   \begin{figure}[ht]
 \begin{center}
\includegraphics{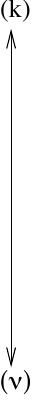}
 \caption{}
 \end{center}
\end{figure}

Let $\Ic$ be an ideal of depth $d(\Ic)$ greater than or equal to one.
We assume that we have constructed the Newton trees of ideals of depths $d<d(\Ic)$.

On one hand we have the graph of the Newton polygon of the ideal $\Ic$. Consider a vertex $v$ on this
graph. It is associated with a face $S$ of the Newton polygon of $\Ic$ with equation
$p_S \ag+q_S\bg=N_S$ and
$$ \ti(\Ic,S)=x^{a_S}y^{b_S}F_{\Ic,S}(x^{q_S},y^{p_S})\big(k_1(x^{q_S},y^{p_S}),\cdots,k_s(x^{q_S},y^{p_S})\big)$$
with $\deg k_i=d_S\geq0$.  We decorate the vertex $v$ with the pair $(N_S,d_S)\in \bn ^2$.

Now we apply the Newton maps $\sigma=\sg_{(p_S,q_S,\mu_i)}$ for each root
$\mu_i$ of the face polynomial. (If the face is dicritical we
already know that the maps $\sg_{(p_S,q_S,\mu)}$ for $\mu$ generic give a
monomial ideal of the form $(x^{N_S})$ and we don't need to perform
those Newton maps.) The transformed ideal $\Ic _{\sg}$ has depth
less than $d(\Ic)$. Then from the induction hypothesis we can
construct the Newton tree of $\Ic _{\sg}$. It has a top arrow
decorated with $N_S$. We delete this arrow and glue the edge on the
vertex $v$. The edge which is glued on the vertex $v$ is a
horizontal edge. Horizontal edges join vertices corresponding to
different Newton polygons and vertical edges join vertices
corresponding to the same Newton polygon. Note that the \lq width\rq\ of the Newton tree of $\Ic$ is precisely its depth $d(\Ic)$.

 %If the depth is one, we define the Newton tree as follows. We consider the graph of the Newton diagram associated to the ideal.  Consider a vertex which is not the last one at the bottom. Since the depth is one, the corresponding face is a dicritical face and the vertex is called {\it dicritical vertex}. It will be drawn in black and decorated with the pair $(N,d)$  where $d$ is the degree of the face.  Consider the last vertex $v$ at the bottom. If the corresponding face is a dicritical face, it is a dicritical vertex drawn in black and decorated with $(N,d)$. If it is not a dicritical vertex then the corresponding face has height one and the transform of the ideal by the corresponding Newton map is principal, generated by $x^k(y+h(x))^m$.
% We erase the decoration $(0)$ of the arrow.
%MUST ADD ARROWS (WITH DECORATIONS)

% Assume that we have constructed the Newton tree of $\Ic$, when $\Ic$ has depth less than or equal to $n-1$. Let now $\Ic$ be an ideal in $\bc [[x,y]]$ of depth $n$. On one hand we can construct the graph of its Newton diagram. We draw in black dicritical vertices corresponding to dicritical faces and decorate them with pairs (N,d). On the other hand we have the Newton trees for all $\Ic _{\sg}$ whose depths are less than or equal
% to $n-1$. We consider only the $\Ic_{\sg_{(p,q,\mu)}}$ coming from faces $S$ of the Newton polygon of
% $\Ic$ and roots $\mu$ of the face polynomial $F_{\Ic,S}(1,X)$. We delete the top arrow of the Newton tree of $\Ic_{\sg_{(p,q,\mu)}}$
% and glue the top edge on the vertex $v$ representing $S$ on the graph of the Newton diagram.

 We explain now how we decorate  the Newton tree. Let $v$ be a vertex on the Newton tree of $\Ic$. If $v$ corresponds to a  face of the Newton polygon of $\Ic$, we say that $v$ has {\it no preceding vertex}
 and we define $\cls (v)=\{v\}$.  If $v$ does not correspond to a face of the Newton polygon of $\Ic$, it corresponds to a face of the Newton polygon of some $\Ic _{\Sg}$. The Newton tree of $\Ic _{\Sg}$ has been glued on a vertex $v_1$ which is called the {\it preceding vertex} of $v$.  We note that the path between one vertex and its preceding vertex contains exactly one horizontal edge but may contain some vertical edges, for example as in Figure 7.

    \begin{figure}[ht]
 \begin{center}
\includegraphics{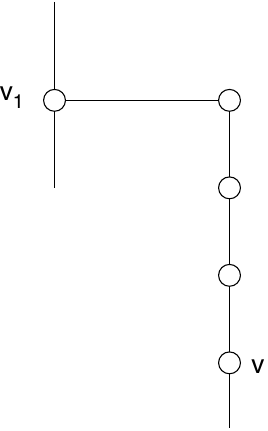}
 \caption{}
 \end{center}
\end{figure}

 If $v_1$ does not correspond to
 a face of the polygon of $\Ic$, we can consider its preceding vertex, and so on. We define
 $\cls(v)=\{v_i, \cdots, v_2,v_1,v\}$ where $v_j, 2\leq j\leq i,$ is the preceding vertex of $v_{j-1}$, and $v_i$ corresponds to a face of the Newton polygon of $\Ic$. The final Newton tree is decorated in the following way.  Let $v$ be a vertex on the Newton tree of $\Ic$. If $\cls(v)=\{v\}$, the decorations near $v$ are not changed. If  $\cls(v)=\{v_i, \cdots, v_2,v_1,v\}$ and if the decorations near $v$ on the Newton tree where $\cls(v)=\{v_{i-1}, \cdots, v_2,v_1,v\}$ are $(m,p)$, then after the gluing on $v_i$ they become $(m+p_iq_ip_{i-1}^2\cdots p_1^2p,p)$. The decorations of the arrows are not changed. We will see later why these decorations are useful.

 The vertices decorated with $(N,d)$ with $d>0$ (corresponding to dicritical faces) are called {\it dicritical vertices}. We denote by $\mathcal{D}_{\Ic}$ the set of dicritical vertices of $\Ic$.

 Note that if the ideal is principal, its Newton tree has all vertices decorated with $(N,0)$. In this case
 we decorate them simply with $(N)$.

% Each vertex on the Newton tree corresponds to a face of the Newton polygon of some $\Ic _{\Sg}$.
 %We have a special set of vertices on the Newton tree of an ideal, that we call {\it dicritical vertices}.
% They correspond to  the dicritical faces of the Newton polygons of the $\Ic _{\Sg}$, which are not empty.
% We denote by $\mathcal{D}_{\Ic}$ the set of dicritical vertices of the Newton tree of the ideal $\Ic$.
% Let $v$ be a vertex on the Newton tree of $\Ic$. Let $\Sg$ be a sequence of Newton maps which produces
% this vertex, and $S$ the face on the Newton polygon of $\Ic_{\Sg}$ which corresponds to $v$. We say that $v$ is a dicritical vertex if $\ti(\Ic_{\Sg},S)$ is not principal, that is, if it needs at least two generators.

%We have
% $$\ti (\Ic,S)(x,y)=(x^N)(F_{\Ic_{\Sg},S}(x^q,y^p))(k_1(x^q,y^p),\cdots, k_s(x^q,y^p)) ,$$
%where  $s\geq 2$, $F_{\Ic_{\Sg},S}\in \bc[x,y]$ with $F_{\Ic_{\Sg},S}(0,y)\neq 0$. The $k_i$ are all homogeneous of the same degree and coprime; we denote this degree by $d_v$ and call it the {\it degree of the dicritical vertex.}
%For any vertex on the Newton tree we denote $d_v=d_S$ and we call it the {\it degree} of the vertex. Thus a vertex $v$ is dicritical if and only if its degree $d_v$ is positive.
%Now we decorate each vertex with the pair $(N_v,d_v)$.

% In the examples later on we draw in black the dicritical vertices.
Usually we do not write the decoration of arrows decorated with $(1)$.

\bigskip
\noindent
{\bf Example 1 {\rm (continued)}.} In Figure 8 we draw the graphs associated with the occurring Newton diagrams, and the resulting Newton tree.

  \begin{figure}[ht]
 \begin{center}
\includegraphics{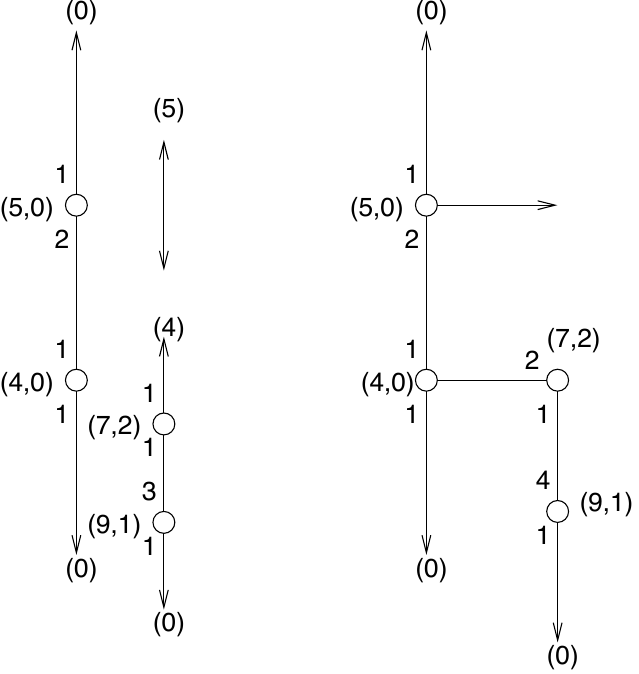}
\caption{}
 \end{center}
 \end{figure}

\bigskip
\noindent
{\bf Example 3.}
We consider in $\bc[[x,y]]$ the ideal
$$\Ic=\left(y^2((x^2+y^3)^2+xy^5)(x^2-y^3),x^8y+x^{12}\right).$$
Its Newton polygon is given in Figure 9.

  \begin{figure}[ht]
 \begin{center}
\includegraphics{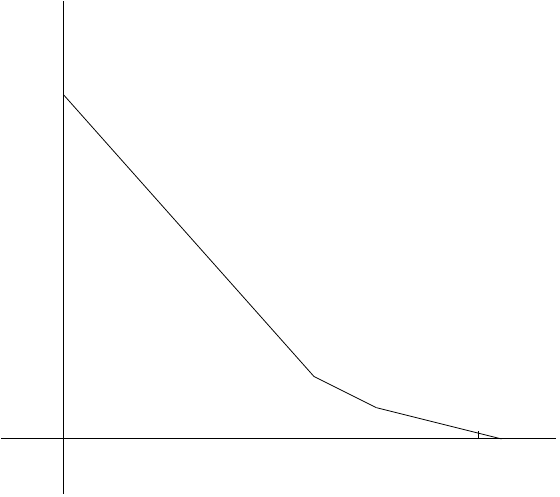}
\caption{}
 \end{center}
 \end{figure}

The faces $S_1$, $S_2$ and $S_3$ have supporting lines with equations $3\alpha+2\beta=22$, $\alpha+2\beta=10$ and $\alpha+4\beta=12$, respectively. The initial ideals of $\Ic$ with respect to these segments are
$$\ti (\Ic,S_1) = \left(y^2(x^2+y^3)^2(x^2-y^3)\right), \ \ \ti (\Ic,S_2) = x^6y (y,x^2) \quad\text{and}\quad \ti (\Ic,S_3)= \left(x^8(y+x^4)\right).$$
Only $S_2$ is dicritical. The face polynomials of $S_1$ and $S_3$ are $F_{\Ic,S_1}=(x+y)^2(y-x)$ and $F_{\Ic,S_3}=y+x$.

First we consider the Newton map $\sg_{(p,q,\mu)}=\sg_{(3,2,1)}$ associated with $S_1$ and $\mu=1$. It is given by the substitution
$$x=x_1^3, \qquad\qquad y=x_1^2(y_1+1).$$
The image ideal $\Ic_1$ is given by
$$
\begin{aligned}
\Ic_1 &= \left(x_1^4 ((x_1^6+x_1^6(y_1+1)^3)^2+x_1^{13}(y_1+1)^5)(x_1^6-x_1^6(y_1+1)^3),
  x_1^{24}(x_1^2(y_1+1)+x_1^{12})\right)\\
&= x_1^{22}(y_1,x_1^4).
\end{aligned}
$$
Its Newton polygon has one face with equation $\ag+4\bg=26$ which  is dicritical (of degree $1$) and has constant face polynomial. Hence we stop this part of the procedure.

Next we consider the Newton map $\sg_{(p,q,\mu)}=\sg_{(3,2,-1)}$ associated with the same face $S_1$ but for $\mu=-1$. It is given by the substitution
$$x=-x_1^3, \qquad\qquad y=x_1^2(y_1-1).$$
The image ideal $\Ic_1$ is given by
$$
\begin{aligned}
\Ic_1 &= \left(x_1^4 ((x_1^6+x_1^6(y_1-1)^3)^2-x_1^{13}(y_1-1)^5)(x_1^6-x_1^6(y_1-1)^3),
  x_1^{24}(x_1^2(y_1-1)+x_1^{12})\right)\\
&=\left(x_1^{22}((y_1^3-3y_1^2+3y_1)^2 - x_1(y_1-1)^5),x_1^{26}\right)\\
&= x_1^{22}(9y_1^2+x_1+ \dots,x_1^4),
\end{aligned}
$$
%where $u \in 1+y_1\bc[[y_1]]$ is a unit.
Its Newton polygon has a unique face $S'$ with equation $2\ag+\bg=46$. It has initial ideal $\ti (\Ic,S')=(9y_1^2+x_1)$. Hence it is not dicritical, and has face polynomial $F_{\Ic,S'}=y+\frac19 x$.

We continue with the Newton map $\sg_{(p,q,\mu)}=\sg_{(2,1,-\frac19)}$ associated with $S'$ and $\mu=-\frac19$. It is given by the substitution
$$x_1=-\frac19x_2^2, \qquad\qquad y_1=x_2(y_2-\frac19).$$
The image ideal $\Ic_2$ is given by
$$
\Ic_2 = x_2^{44} \left(9x_2^2(y_2-\frac19)^2-\frac19x_2^2+ \dots,x_2^8\right) = x_2^{46}(54y_2+x_2+ \dots,x_2^6).
$$
Its Newton polygon has only one face with equation $\ag+\bg=47$ which is not dicritical.
We perform the following Newton maps:
$$x_2=x_3, \qquad\qquad y_2=x_3(y_3-1/54),$$
 $$x_3=x_4, \qquad\qquad y_3=x_4(y_4-55/17496),$$
$$x_4=x_5, \qquad\qquad y_4=x_5(y_5-10/19683),$$
$$x_5=x_6, \qquad\qquad y_5=x_6(y_6-7/93312),$$
$$x_6=x_7, \qquad\qquad y_6=x_7(y_7-91/9565938),$$
and we arrive at a dicritical face of degree 1.

%(of degree $1$) and has constant face polynomial. Hence we stop the part of the procedure associated to $S_1$.

Since the face $S_2$ is dicritical (of degree $1$) with constant face polynomial we do not handle it further.
Finally we consider the Newton map $\sg_{(p,q,\mu)}=\sg_{(1,4,-1)}$ associated with $S_3$ and $\mu=-1$. It is given by the substitution
$$x=x_1, \qquad\qquad y=x_1^4(y_1-1).$$
The image ideal $\Ic_1$ is given by
$$
\begin{aligned}
\Ic_1 &= \left(x_1^8 ((x_1^2+x_1^{12}(y_1-1)^3)^2+x_1^{21}(y_1-1)^5)(x_1^2-x_1^{12}(y_1-1)^3),
  x_1^{8}(x_1^4y_1)\right)\\
&=(x_1^{14},x_1^{12}y_1)\\
&= x_1^{12}(x_1^2,y_1).
\end{aligned}
$$
Its Newton polygon has only one face with equation $\ag+2\bg=14$ which  is dicritical (of degree $1$) and has constant face polynomial. Hence we stop the procedure.
The depth of $\Ic$ is $8$; its Newton tree is given in Figure 10.

\bigskip
  \begin{figure}[ht]
 \begin{center}
\includegraphics{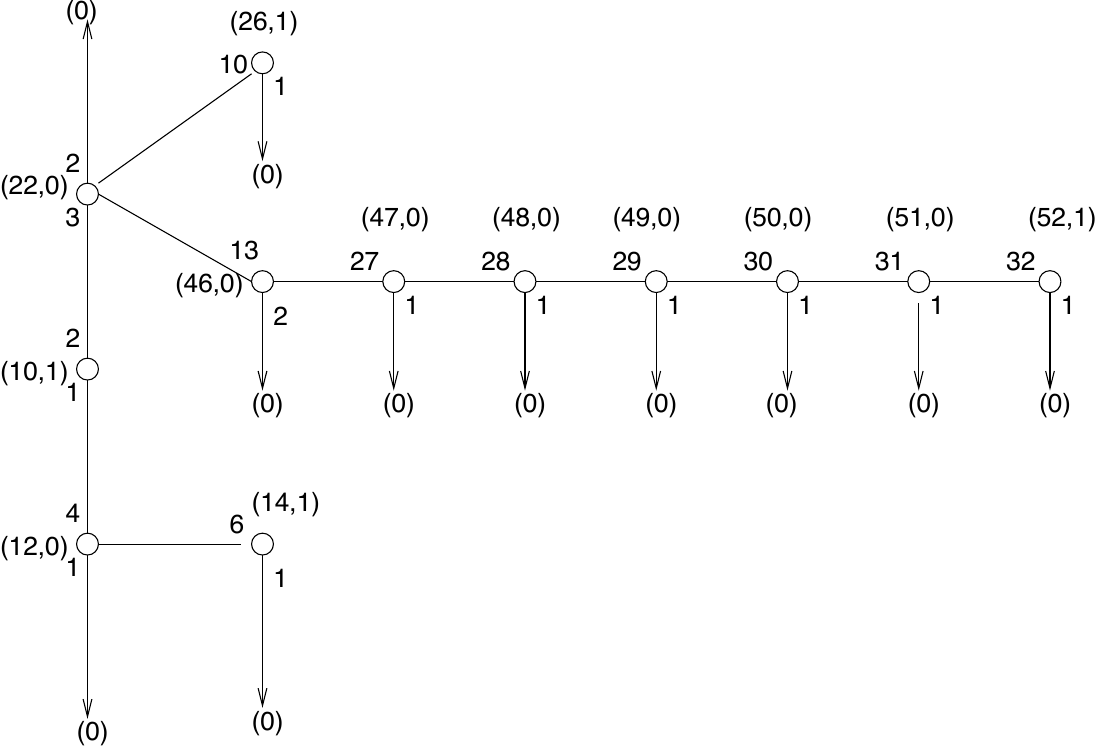}
\caption{}
 \end{center}
 \end{figure}

\bigskip

\bigskip
\noindent
{\bf Example 4.}
We consider in $\bc[[x,y]]$ the ideal
$$\Ic=\left((x-y)^2x^3,(x-y)^2y^3,(x-y)x^6\right) .$$ Its Newton polygon has
 one face $S$ whose supporting line has equation $\alpha+\beta=5$, and with initial ideal
$$\ti (\Ic,S)=\left((x-y)^2x^3,(x-y)^2y^3\right) .$$
Hence it is dicritical (of degree $3$) and its face polynomial is $(x-y)^2$.

We consider the Newton map $\sg_{(p,q,\mu)}=\sg_{(1,1,1)}$ associated with $S_1$ and $\mu=1$. It is given by the substitution
$$x=x_1, \qquad\qquad y=x_1(y_1+1).$$
The image ideal $\Ic_1$ is given by
$$
\begin{aligned}
\Ic_1 &= \left(x_1^5y_1^2, x_1^5y_1^2(y_1+1)^3,x_1^7y_1\right)\\
&= x_1^5y_1(y_1,x_1^2).
\end{aligned}
$$
Its Newton polygon has one  face with equation $\ag+2\bg=9$ which is  is dicritical (of degree $1$) and has constant face polynomial. Hence we stop the procedure.
The depth of $\Ic$ is $2$; its Newton tree is given in Figure 11.

\bigskip
 \begin{figure}[ht]
 \begin{center}
\includegraphics{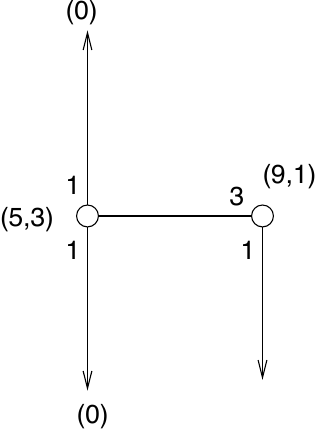}
\caption{}
 \end{center}
 \end{figure}
\bigskip
\subsection{Combinatorial properties of  Newton trees}

\begin{prop}
Consider a Newton tree of an ideal. If $v_0$ is the preceding vertex
of $v$, decorated respectively by $(q_0,p_0)$ and $(q,p)$, we have
$$q=p_0q_0p+\tilde{m},$$
where ($\tilde{m},p)$ are the decorations of $v$ on the Newton tree where $\mathcal{S}(v)=\{v\}$.
\end{prop}

\begin{proof}
We work by induction on the number of elements of $ \mathcal{S}(v)$.
If $ \mathcal{S}(v)=\{v_0,v\}$, then $q=p_0q_0p+\tilde{m}$ is the definition of $q$.
Assume that $ \mathcal{S}(v)=\{v_i,v_{i-1},\cdots, v_0,v\}$, and that on the Newton tree where
$ \mathcal{S}(v)=\{v_{i-1},\cdots, v_0,v\}$, the decorations $(\tilde{q}_0,p_0)$ of $v_0$ and $(\tilde{q},p)$ of $v$ satisfy $\tilde{q}=p_0\tilde{q}_0p+\tilde{m}$.

We have (by definition)
$$q_0=p_iq_ip_{i-1}^2\cdots p_1^2p_0+\tilde{q}_0,$$
$$q=p_iq_ip_{i-1}^2\cdots p_1^2p_0^2p+\tilde{q},$$
and hence
$$q=p_iq_ip_{i-1}^2\cdots p_1^2p_0^2p+p_0\tilde{q}_0p+\tilde{m}
=p_0p(p_iq_ip_{i-1}^2\cdots p_1^2p_0+\tilde{q}_0)+\tilde{m}
= p_0pq_0+\tilde{m}.$$

\end{proof}

\begin{definition}
Consider a path on a Newton tree. We say that a number is {\it adjacent} to this path if it is a decoration near a vertex $v$ on the path, on an edge  issued from the vertex $v$, not belonging to the path.
\end{definition}

If $v$ and $w$ are vertices and $f$ an arrow on a Newton tree, we denote by $\rho_{v,w}$ (resp. $\rho_{v,f}$) the product of the numbers adjacent to the path between the vertices $v$ and $w$ (resp. the path between the vertex $v$ and the arrow $f$).

\begin{prop}\label{decorationN}
The decoration $(N_v)$ of a vertex $v$ on a Newton tree of an ideal $\Ic$ is  equal to
$$N_v= \sum_{f\in\mathcal{F}}\rho_{v,f}m(f) +\sum_{w\in \mathcal{D}_{\Ic}} \rho_{v,w}d_w ,$$
where $\mathcal{F}$ denotes the set of arrows of the Newton tree, and $m(f)$ the decoration of the arrow $f$.
\end{prop}

\begin{proof}
If we consider an arrow $f$ on a tree, it is attached to a vertex
$w$. We will denote $\mathcal{S}(f)= \mathcal{S}(w)$. Let $S$ be a
face of some Newton polygon. We denote by $\vert S\vert _1$ the
width of $S$ and by $\vert S\vert _2$ its height. To prove the
proposition we use the following lemma.

\begin{lemma}\label{hauteur} Let $v$ be a vertex on the Newton tree of an ideal $\Ic$ and $S_v$ the corresponding face of some Newton polygon. We have
%$$\vert S_v\vert _1=1/p(\sum_{f\in \mathcal{F}\vert v\in S(f)}\rho_{v,f}m(f)+\sum_{w\in \mathcal{D}_{\Ic}\vert v\in \mathcal{S}(w)}\rho_{v,w}d_w)$$
% and
$$\vert S_v\vert _2=\frac1q \left(\sum_{f\in \mathcal{F}\vert v\in S(f)}\rho_{v,f}m(f)+\sum_{w\in \mathcal{D}_{\Ic}\vert v\in \mathcal{S}(w)}\rho_{v,w}d_w\right),$$
where $(q,p)$ are the decorations near $v$.
\end{lemma}

First we prove the lemma.
\begin{proof}
%We need only to prove the second equality.
We use induction on the depth. Let $\Ic$ be an ideal with
$d(\Ic)=1$. Let $S$ be a face of its Newton polygon corresponding to
a vertex $v$ on the (in this case vertical) Newton tree of $\Ic$
with
$$ \ti(\Ic,S)=\big(x^{a_S}y^{b_S}F_{\Ic,S}(x^{q_S},y^{p_S})\big)$$
or
$$ \ti(\Ic,S)=x^{a_S}y^{b_S}F_{\Ic,S}(x^{q_S},y^{p_S})\left(k_1(x^{q_S},y^{p_S}),\cdots,k_s(x^{q_S},y^{p_S})\right)$$
as in (\ref{initialideal1}) or (\ref{initialideal2}), with
$F_{\Ic,S}(x,y)=\prod_{i}(y-\mg_{S,i}x)^{\nu_{S,i}}$. Then
$$\vert S\vert _2=p_S\, (d_S+\sum_i \nu_{S,i}) .$$
For each arrow $f_i$ of multiplicity $m(f_i)=\nu_{S,i}$ we have $\rho_{v,f_i}=q_Sp_S=\rho_{v,v}$.
Hence
$$\vert S\vert _2=\frac1{q_S} \left(\rho_{v,v}d_S+\sum_i \rho_{v,f_i}m(f_i)\right) .$$

Now we assume that $\Ic$ is an ideal of depth $d(\Ic)>1$ and that the lemma is true for depth $d<d(\Ic)$.
Let $S$ be a face of the Newton polygon of $\Ic$ corresponding to a vertex $v$ on the Newton tree.
We have again
$$\vert S\vert _2=p_S\, (d_S+\sum_i \nu_{S,i}) ,$$
but now $ \nu_{S,i}=h_i+l_i$, where $h_i$ is the height of the
Newton polygon of the ideal $\Ic_{\sg_i}=\Ic_{\sg_{(p_S,q_S,\mu_i)}}$ and $l_i$ the maximal
power of $y_1$ which divides $\Ic_{\sg_i}$. Denote by $\tilde{f}_i$
the branch corresponding to $y_1^{l_i}$. We have
$\rho_{v,\tilde{f}_i}=p_Sq_S=\rho_{v,v}$. Now we apply the induction
hypothesis to the faces of the Newton polygon of the $\Ic_{\sg_i}$
and a simple computation on the $\rho$'s yields the formula.
\end{proof}

We prove the proposition by induction on the number of elements of
$\mathcal{S}(v)$. If $\mathcal{S}(v)=\{v\}$, then $v$ is on the first
Newton polygon. Let $S$ be the face of the Newton polygon
corresponding to $v$, with $p\alpha+q\beta=N_v$ the equation of the
supporting line and $(\alpha', \beta')$ the origin of $S$. We have
$$\alpha'=\alpha_0+\sum _{i} \vert S_i\vert _1, \qquad \beta' =\vert S\vert _2+\sum _j\vert S_j\vert _2 +\beta_m,$$
where $i$ ranges over the indices of faces before  $S$, and
$j$ over the indices of faces after $S$. Since $N_v=p\alpha'
+q\beta'$ the property for $\mathcal{S}_v=\{v\}$ follows from Lemma
\ref{hauteur}, applied to all faces of the first Newton polygon, and
the fact that $\vert S_i\vert _1 = \frac{q_i}{p_i} \vert S_i\vert
_2$. (Remember that a positive value of $\alpha_0$ or $\beta_m$
corresponds to a decoration of an arrow.)

Now let $\mathcal{S}_v=\{v_{k},\cdots ,v_1,v_0, v\}$ and hence
$\mathcal{S}_{v_0}=\{v_{k},\cdots ,v_1,v_0\}$. We apply the
induction hypothesis to $v_0$ and write the decoration $(N_0)$ of
$v_0$ as the sum of the contributions of all arrows and dicriticals
\lq on the right\rq\ of $v_0$ and the sum $N_0^1$ of all other
contributions. More precisely
$$N_0=\tilde{h}(\Delta)p_0q_0+N_0^1 ,$$
where $\Delta$ is the Newton polygon to which belongs the face $S$
corresponding to $v$ and
$$\tilde{h}(\Delta)=\sum _i \vert S_i\vert _2+ \vert S\vert _2+\sum _j\vert S_j\vert _2 +\beta'_m.$$
Here we use again Lemma \ref{hauteur}, $i$ and $j$ are as before and $\beta'_m$ is the
decoration of the bottom arrow in the diagram of $\Delta$.
Using the equation $p\alpha+\tilde{m}\beta=N$ of the line supporting
$S$, we have
$$
\begin{aligned}
N&=p\left(N_0+\sum _i \vert S_i\vert _1)+ \tilde{m}(\vert S\vert _2+\sum _j\vert S_j\vert _2 +\beta'_{m}\right) \\
&=p\left(N_0^1+\tilde{h}(\Delta)p_0q_0\right)+p\sum _i \vert S_i\vert
_1 +\tilde{m}\left(\vert S\vert _2+\sum _j\vert S_j\vert _2
+\beta'_{m}\right)\\
&=pN_0^1+(\tilde{m}+p_0q_0p)\left(\vert S\vert _2+\sum _j\vert S_j\vert
_2 +\beta'_{m}\right) +p \left(\sum_i \frac{\vert S_i\vert _2}{p_i}(\tilde
{m_i}+p_0q_0p_i)\right),
\end{aligned}
$$
where $p_i\alpha+\tilde{m}_i\beta=N_i$ is the equation of the line
supporting $S_i$.
 Using Lemma \ref{hauteur} and the fact that
$\tilde{m}+p_0q_0p=q$ and  $\tilde {m}_i+p_0q_0p_i= q_i$, we obtain
the result.
\end{proof}

\begin{remark}
 From this proposition we see that when the tree is constructed, the only decorations of the  vertices which are  needed  are the $d_v$, because the $N_v$ can be computed from the $d_v$. But, anyway,  we have to keep in mind that we know the $N_v$ when $v$ is constructed.
\end{remark}

%If $v$ and $w$ are vertices and $f$ an arrow on a Newton tree, we denote by $\rho_{v,w}$ (resp. $\rho_{v,f}$) the product of the numbers adjacent to the path between the vertices $v$ and $w$ (resp. the path between the vertex $v$ and the arrow $f$).

%\begin{prop}
%The decoration $(N_v)$ of a vertex $v$ on a Newton tree of an ideal $\Ic$ is  equal to
%$$N_v= \sum_{f\in\mathcal{F}}\rho_{v,f}m(f) +\sum_{w\in \mathcal{D}_{\Ic}} \rho_{v,w}d_w ,$$
%where $\mathcal{F}$ denotes the set of arrows of the Newton tree, and $m(f)$ the decoration of the arrow $f$.
%\end{prop}

%\begin{proof}
%If we consider an arrow $f$ on a tree, it is attached to a vertex $v$. We will denote by $\mathcal{S}(f)=
%\mathcal{S}(v)$. Let $S_v$ be a face of some Newton polygon. We denote by $\vert S_v\vert _1$
%the width of $S_v$ and $\vert S_v\vert _2$ its height. To prove the proposition we need the lemma

%\begin{lemma}\label{hauteur}
%$$\vert S_v\vert _1=1/p(\sum_{f\in \mathcal{F}\vert v\in S(f)}\rho_{v,f}m(f)+\sum_{w\in \mathcal{D}_{\Ic}\vert v\in \mathcal{S}(w)}\rho_{v,w}d_w)$$
% and
%$$\vert S_v\vert _2=1/q(\sum_{f\in \mathcal{F}\vert v\in S(f)}\rho_{v,f}m(f)+\sum_{w\in \mathcal{D}_{\Ic}\vert v\in \mathcal{S}(w)}\rho_{v,w}d_w)$$
%\end{lemma}

\subsection{Comparison of the Newton tree of an ideal and the Newton tree of a generic curve of the ideal}

 In the sequel we will consider a {\em generic curve} of a non-principal ideal, given by a $\bc$-linear combination with generic coefficients of the generators of the ideal. This notion depends in fact on the chosen generators of the ideal, but its properties with respect to Newton trees are independent of that choice.

\begin{prop}\label{ideal-gencurve}
Let $\Ic =(f_1,\cdots, f_r)$ be a non-principal ideal in $\bc[[x,y]]$. %We assume that the support of $\Ic$ is $\{0\}$.
%The set of vertices of the Newton tree of $\Ic$ is the union of the set of vertices of the Newton tree of a sufficiently generic curve of $\Ic$  and the set of dicritical vertices. The vertices belonging to both trees bear the same decorations in their common vertices.
The Newton tree of a generic curve of  $\Ic$ is obtained from the
Newton tree of $\Ic$ by adding to each dicritical vertex $v$ exactly
$d_v$ arrows with multiplicity one. The decorations of the edges and
the decorations $N$ of the vertices are the same.
\end{prop}

\begin{proof}
Let $f=\lbg _1f_1+\cdots +\lbg_rf_r\in \bc [\lbg_1,\cdots,
\lbg_r][[x,y]]$. For generic values of $\lambda_1,\cdots, \lambda_r$
we have $\tS f=\tS \, \Ic $ and then $\nc(f)=\nc(\Ic)$.

Consider a face $S$ of the Newton polygon  $\nc (\Ic)$.
\begin{itemize}

\item If the corresponding vertex is not a dicritical vertex, then there is a non-constant polynomial $h$ such that
$$\ti (\Ic,S)=(h)$$
and
$$\ti(f,S)=(\lbg_1u_1+\cdots+\lbg_s u_s)h$$
with $u_i\in \bc$ for all $i\in \{1,\cdots,s\}$ and $u_i\neq 0$ for
at least one $i$.

\item If  the corresponding vertex is a dicritical vertex, then $\ti (\Ic,S)$ is of the form
$$\ti (\Ic,S)=(h)\big(k_1(x^q,y^p),\cdots, k_s(x^q,y^p)\big)$$
with $s\geq 2$, where $k_1,\cdots, k_s$ are homogeneous polynomials,
coprime and of the same degree $d_v \geq 1$. In this case
$$\ti(f,S)=h\big(\lambda_1 k_1(x^q,y^p)+\cdots+ \lambda_s k_s(x^q,y^p)\big).$$
The polynomial $\lambda_1 k_1(x,y)+\cdots+ \lambda_s k_s(x,y)$  is
homogeneous of degree $d_v$, and factorizes in $d_v$ factors of
multiplicity $1$ for generic $(\lbg_1,\cdots, \lbg_s)\in\bc^s $.
Hence, when $v$ is a dicritical vertex of the Newton tree of $\Ic$,
there are $d_v$ arrows connected to $v$ on the Newton tree of the
generic curve.
\end{itemize}
%\begin{definition}
%The number of roots of $F_v$ is called the degree $d_v$ of the dicritical vertex $v$.
%\end{definition}

%\begin{rem}
%However, if
%$d_v=1$ it might happen that after changing to (very) good coordinates, we obtain a monomial for the transform of $f$. In that case the dicritical vertex on the tree of $\Ic$ is not a vertex on the Newton tree of the generic curve.

At any stage of the Newton algorithm for $\Ic$  we have the same two
cases. If $v$ is a dicritical vertex of the Newton tree of $\Ic$ and
$(h)=\left(x^k (y+h(x))^{\nu}\right)$, the algorithm stops and $v$
is an end of the Newton tree of $\Ic$ with an arrow of multiplicity
$\nu$. Otherwise we perform a Newton map and we go on.
%\end{rem}

The assertion on the decorations is immediate from the definition of
the decorations.

\end{proof}

\bigskip
\noindent {\bf Examples 3 and 4 {\rm (continued)}.}

The Newton tree of the generic curves are given in Figure 12 and 13,
respectively.

  \begin{figure}[ht]
 \begin{center}
\includegraphics{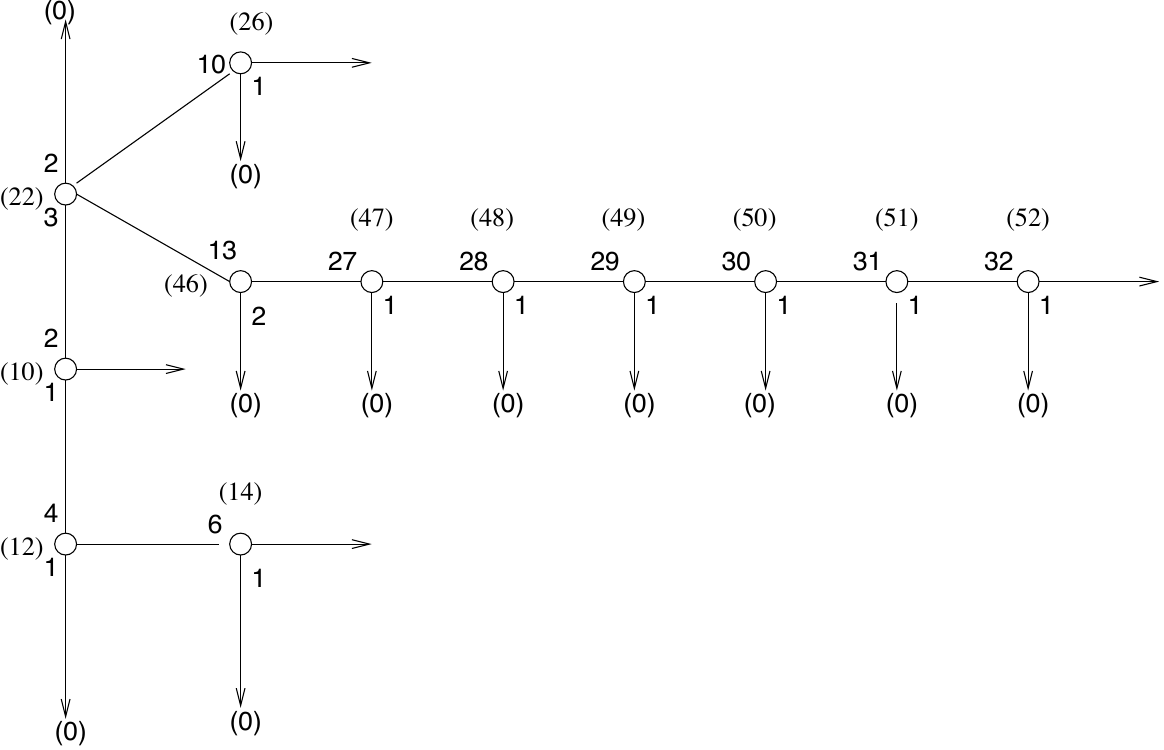}
\caption{}
 \end{center}
 \end{figure}

\bigskip

  \begin{figure}[ht]
 \begin{center}
\includegraphics{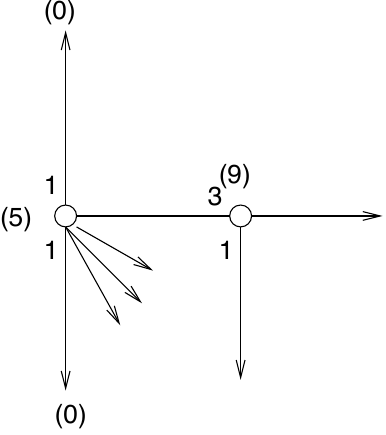}
\caption{}
 \end{center}
 \end{figure}

% \bigskip
%\noindent
%{\bf Example 5}

%In this example we consider two ideals with the same Newton tree, but different Newton process.

%Let $$\Ic_1=(x^4+2x^2y^3+x^5, xy^5+x^2y^6,y^7+xy^6)$$
%and
%$$\Ic_2=(x^4+3x^2y^3,x^3y^2,y^7)$$
%The Newton tree is

%  \begin{figure}[ht]
% \begin{center}
%\includegraphics{descv30.pdf}
%\caption{}
% \end{center}
% \end{figure}

% The Newton process of $\Ic_1$ is
% $$\{(\sg_{(3,2,2)},\sg_{(1,1,\mu)},1),(\sg_{(2,1,\mu)},2)\}$$
% and the Newton process of $\Ic_2$ is
 %$$\{(\sg_{(3,2,3)},\sg_{(1,1,\mu)},1),(\sg_{(2,1,\mu)},2)\}$$

\section{Newton process of an ideal}

\subsection{Description of the Newton process}

The Newton process of an ideal is a set of pairs $(\Sg;Z)$, where
$\Sg$ is an ordered sequence of Newton maps (maybe empty) of length
$n_{\Sg}\leq d(\Ic) $  and either $Z\in \bn^*$ or
$Z=(y+h(x))^{\nu}$, with $h(x)\in x\bc [[x]]$ and $\nu\in\bn^*$.

\smallskip
Let $\Ic$ be a non-trivial ideal in $\bc[[x,y]]$. We assume that $\Ic$ is not
divisible by a power of $x$ (otherwise we just remember the factor
$(x^k)$).

 If $d(\Ic)=0$, then $\Ic =(y+h(x))^{\nu}$ with $h(x)\in
x\bc [[x]]$ and $\nu\in\bn^*$.  By definition the Newton process of
$\Ic$ is $\{\big(\emptyset; (y+h(x))^{\nu}\big)\}$.

% If $d(\Ic)=1$, let $S$
%\ag+q_S\bg=N_S$ and
%$$ \ti(\Ic,S)=x^{a_S}y^{b_S}F_{\Ic,S}(x^{q_S},y^{p_S})\left(k_1(x^{q_S},y^{p_S}),\cdots,k_s(x^{q_S},y^{p_S})\right)$$
%with $\deg k_i=d_S\geq0$.  We define
%$$\tilde {\Sg}_S=\{(\sg_{(p_S,q_S,\mu)},d_S),(\sg_{(p_S,q_S,\mu_i)},(y+h_i(x))^{\nu_i})\}$$
%where $\mu \in \bc^*$ is generic, and $\mu_i$ runs through the roots of the face polynomial and $\Ic_{\sg_{(p_S,q_S,\mu_i)}}=x^k(y+h_i(x))^{\nu_i}$ (note that here $x,y$ are the coordinates after a Newton map.)
%Then we define the Newton process of $\Ic$ to be the set
%$$\{\tilde {\Sg}_S, (\emptyset; y^l)\}$$
%where $S$ runs through the Newton polygon of $\Ic$ and $\Ic$ is divisible by $y^l$ (we forget about
%the last term when $\Ic$ is not divisible by a power of $y$.)

If $d(\Ic)>0$, we say that a vertex on the Newton tree of $\Ic$ is
{\it final} if either it is dicritical or it has an arrow with
positive multiplicity attached (both can happen simultaneously). In
the case of depth $1$ all vertices are final. In the case of a
principal ideal all final vertices have at least one arrow attached.

%Assume that $\Ic$ has depth greater than $1$.
  Let $v$ be a final vertex of its Newton tree.
We have $\mathcal{S}(v)=\{v_j,\cdots, v_1,v\}$, where $v_j$ is the
preceding vertex of $v_{j-1}$, \dots, and $v_1$ is the preceding vertex
of $v$. We denote by $\sg_j$ the Newton map which produces $v_{j-1}$
from $v_j$, \dots, and by $\sg_1$ the Newton map which produces $v$
from $v_1$. Now $v$ corresponds to a face $S$ of the Newton polygon
of some ideal $\Ic_\Sigma$. Let the equation of the line supporting
$S$ be $p_S\ag+q_S\bg=N_S$ and
$$ \ti(\Ic,S)=x^{a_S}y^{b_S}F_{\Ic,S}(x^{q_S},y^{p_S})\big(k_1(x^{q_S},y^{p_S}),\cdots,k_s(x^{q_S},y^{p_S})\big)$$
with $\deg k_i=d_S\geq0$ as in (\ref{initialideal1}) or
(\ref{initialideal2}).
 We denote
$$\tilde {\Sg}_v=\{(\sg_j,\cdots, \sg_1,\sg_{(p_S,q_S,\mu)};d_S)\}\cup\{\big(\sg_j,\cdots, \sg_1,\sg_{(p_S,q_S,\mu_i)};(y+h_i(x))^{\nu_i}\big)\mid i\in I\},$$
where $\mu \in \bc^*$ is generic,  $\{\mu_i\mid i\in I\}$ are  the
roots of the face polynomial of $S$ and
$\Ic_{\sg_{(p_S,q_S,\mu_i)}}=\left(x^{k_i}(y+h_i(x))^{\nu_i}\right)$
(here we keep using the coordinates $x,y$ after any Newton map).

We then define the  Newton process of $\Ic$ to be the set
$$\cup_v \tilde {\Sg}_v \cup \{(\emptyset; y^{\beta_m})\},$$
where $v$ runs through the final vertices  of $\Ic$  (we forget
about the last term when $\Ic$ is not divisible by a power of $y$.)

\bigskip
\noindent {\bf Example 1 {\rm (continued)}.} There are three final
vertices: the vertex $v_1$ decorated with $(5,0)$ and
$$\tilde{\Sg}_{v_1}=\{(\sg_{(2,1,3)};y)\},$$
the vertex $v_2$ decorated with $(7,2)$ and
$$\tilde{\Sg}_{v_2}=\{(\sg_{(1,1,-1)}, \sg_{(1,1,\mu)};2)\},$$
and the vertex $v_3$ decorated with $(9,1)$ and
$$\tilde{\Sg}_{v_3}=\{(\sg_{(1,1,-1)}, \sg_{(1,3,\mu)};1)\}.$$

Finally the Newton process of $\Ic$ is
$$\{(\sg_{(2,1,3)};y),(\sg_{(1,1,-1)}, \sg_{(1,1,\mu)};2),  (\sg_{(1,1,-1)}, \sg_{(1,3,\mu)};1)\}.$$

\bigskip
\noindent {\bf Example 3 {\rm (continued)}.} The Newton process is
$$\{ (\sg_{(3,2,1)}, \sg_{(1,4,\mu)};1), $$
$$\begin{aligned} (\sg_{(3,2,-1)},\sg_{(2,1,-1/9)}, \sg_{(1,1,-1/54)},\sg_{(1,1,-55/17496)},&\sg_{(1,1,-10/19683)}, \\ &\sg_{(1,1,-7/93312)},\sg_{(1,1,-91/9565938)},\sg_{(1,1,\mu)};1),
\end{aligned}$$
$$(\sg_{(1,2,\mu)};1),(\sg_{(1,4,-1)}, \sg_{(1,2,\mu)};1)\}.$$

\bigskip
\noindent {\bf Example 4 {\rm (continued)}.} The two vertices are
final. The Newton process is
$$\{(\sg_{(1,1,\mu)};3),(\sg_{(1,1,1)}, \sg_{(1,2,\mu)};1),(\sg_{(1,1,1)};y)\}.$$

%\bigskip
%In the case of a principal ideal $(f)$ the Newton process is the
%same as the Puiseux expansion of the branches of $f$.

\medskip
 \begin{prop}\label{processofproduct}
Let $\Ic=\Ic_1\Ic_2$ be a non-trivial ideal in $\bc[[x,y]]$. One can find the
Newton process of $\Ic$ from the Newton process of $\Ic_1$ and the
Newton process of $\Ic_2$. One applies the following rule to each
element of the union of the Newton processes of $\Ic_1$ and $\Ic_2$.
 \begin{enumerate}
 \item Let $$A_1=(\sg_{(p_1,q_1,\mu_1)},\cdots, \sg_{(p_n,q_n,\mu_n)};(y+h(x))^{\nu_1})$$
 be an element of the Newton process of $\Ic_1$. If there is no element
 $$A_2=(\sg_{(p_1,q_1,\mu_1)},\cdots, \sg_{(p_n,q_n,\mu_n)};(y+h(x))^{\nu_2})$$
 in the Newton process of $\Ic_2$, then $A_1$ is in the Newton process of $\Ic$.
 If there is such an element, then
 $$A=(\sg_{(p_1,q_1,\mu_1)},\cdots, \sg_{(p_n,q_n,\mu_n)};(y+h(x))^{\nu_1+\nu_2})$$
 is in the Newton process of $\Ic$.

 \item Let $$B_1=(\sg_{(p_1,q_1,\mu_1)},\cdots, \sg_{(p_n,q_n,\mu)};k_1)$$
 be an element of the Newton process of $\Ic_1$. If there is no element
 $$B_2=(\sg_{(p_1,q_1,\mu_1)},\cdots, \sg_{(p_n,q_n,\mu)};k_2)$$
 in the Newton process of $\Ic_2$, then $B_1$ is in the Newton process of $\Ic$.
 If there is such an element, then
 $$B=(\sg_{(p_1,q_1,\mu_1)},\cdots, \sg_{(p_n,q_n,\mu)};k_1+k_2)$$
 is in the Newton process of $\Ic$.
 \end{enumerate}
  There are no other elements in the Newton process of $\Ic$.
 \end{prop}

 \begin{proof}
Let $(p,q)$ be a pair of positive natural numbers prime to each
other, and $l$ a line with equation $p\ag+q\bg=N$ which hits the
Newton polygon of $\Ic$. The initial ideal of $\Ic$ with respect to
$l$ is the ideal generated by the $\ti(g,l)$, where $g$ runs through
a set of generators of $\Ic$.

Certainly $\ti (\Ic,l)\neq (0)$ and $\ti(\Ic
,l)=\ti(\Ic_1,l_1)\ti(\Ic_2,l_2)$, where $l_1$ and $l_2$ are lines
parallel to $l$ which hit respectively the Newton polygon of $\Ic_1$
and the Newton polygon of $\Ic_2$. The line $l$ does not support a
face of the Newton polygon of $\Ic$ if and only if $\ti (\Ic,l)$ is
a monomial, which is equivalent to the fact that $\ti(\Ic_1,l_1)$
and $\ti(\Ic_2,l_2)$ are monomials and hence to the fact that $l_1$
and $l_2$ do not support faces of the Newton polygon of $\Ic_1$ and
$\Ic_2$, respectively.

On the other hand the line $l$ supports a face $S$ of the Newton
polygon of $\Ic$ if and only if $l_1$ supports a face $S_1$ of the
Newton polygon of $\Ic_1$ and/or $l_2$ supports a face $S_2$ of the
Newton polygon of $\Ic_2$. Denote the (non-monomial) ideal
$\ti(\Ic,S)$ as usual as
 $$
 \ti(\Ic,S)=x^{a}y^{b}F_{\Ic,S}(x^{q},y^{p})\Ic',
 $$
 where $\Ic'$ has finite codimension.
  Then $\ti(\Ic_1,l_1)$ and $\ti(\Ic_2,l_2)$ are of the form
 $$\ti(\Ic_1,l_1)=x^{a_1}y^{b_1}F_{\Ic_1,l_1}(x^{q},y^{p})\Ic_1'$$
 and
 $$\ti(\Ic_2,l_2)=x^{a_2}y^{b_2}F_{\Ic_2,l_2}(x^{q},y^{p})\Ic_2',$$
where $F_{\Ic_1,l_1},F_{\Ic_2,l_2}$ are polynomials and
$\Ic_1',\Ic_2'$ ideals of finite codimension satisfying
$F_{\Ic,S} = F_{\Ic_1,l_1}F_{\Ic_2,l_2}$, and $\Ic'=\Ic_1'\Ic_2'$.

 If $S$ is a dicritical face, then $S_1$ and/or $S_2$ is a dicritical face and $d_S=d_{S_1}+d_{S_2}$.
 A complex number $\mu$ is a root of $F_{\Ic,S}$ if and only if it is a root of $F_{\Ic_1,l_1}$ and/or $F_{\Ic_2,l_2}$, and the multiplicity of $\mu$ in $F_{\Ic,S}$ is the sum of the multiplicities of $\mu$ in $F_{\Ic_1,l_1}$ and  in $F_{\Ic_2,l_2}$.

Moreover we have
$$\sg_{(p,q,\mu)}(\Ic)=\sg_{(p,q,\mu)}(\Ic_1)\sg_{(p,q,\mu)}(\Ic_2).$$
We conclude by induction on the depth.
 \end{proof}

  \begin{cor}
Let $\Ic=(f)\Ic_1$ be an ideal in $\bc[[x,y]]$, with $\Ic_1$ an
ideal of finite codimension. Then the Newton process of $\Ic$ is the
union of the Newton process of $(f)$ and the Newton process of
$\Ic_1$.
 \end{cor}

\subsection{Newton process versus Newton tree}

We explain how to recover the Newton tree from the Newton process.
If we have only one element in the Newton process of the form
$$(\sg_{(p_1,q_1,\mu_1)},\cdots,\sg_{(p_n,\tilde{m}_n,\mu_n)},\sg_{(p_{n+1},\tilde{m}_{n+1},\mu)};k),$$
then the Newton tree is as in Figure 14.

\bigskip
 \begin{figure}[ht]
 \begin{center}
\includegraphics{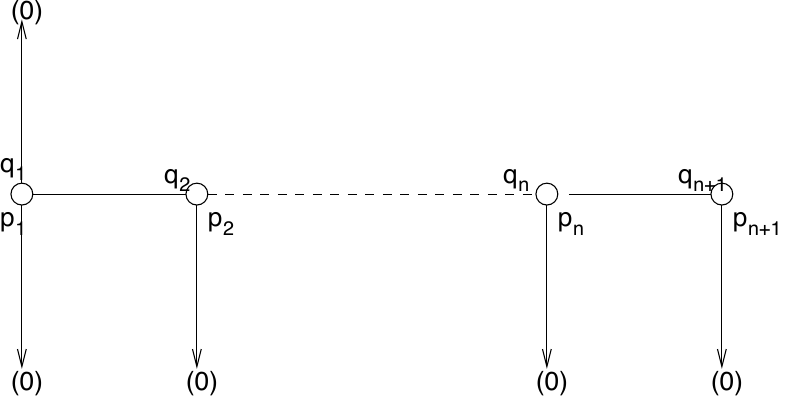}
\caption{}
 \end{center}
 \end{figure}
\bigskip
\noindent Here $q_2=p_1q_1p_2+\tilde{m}_2,\cdots,
q_{n+1}=p_nq_np_{n+1}+\tilde{m}_{n+1}$. We did not write the
decorations of the vertices. We can compute $N$ from Proposition
\ref{decorationN}, and the only vertex with $d>0$ is the last one on
the right where $d=k$.

If the unique  element in the Newton process is of the form
$$\big(\sg_{(p_1,q_1,\mu_1)},\cdots,\sg_{(p_n,\tilde{m}_n,\mu_n)},\sg_{(p_{n+1},\tilde{m}_{n+1},\mu)};(y+h(x))^{\nu}\big),$$
the Newton tree is as in Figure 15.

\bigskip
 \begin{figure}[ht]
 \begin{center}
\includegraphics{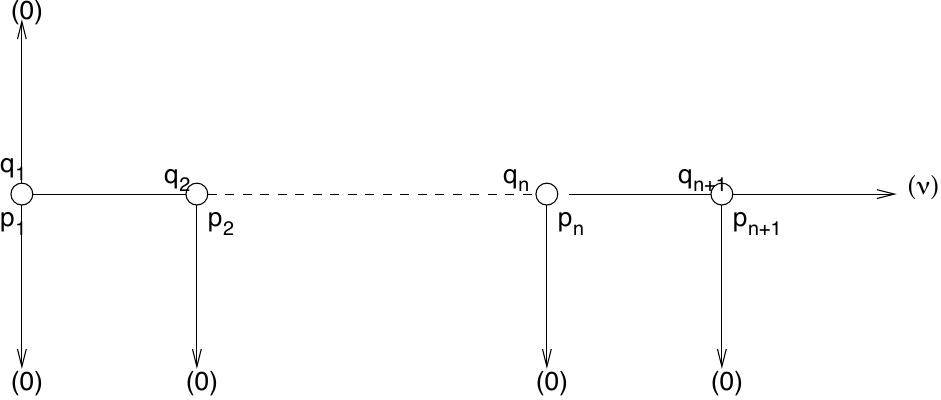}
\caption{}
 \end{center}
 \end{figure}
\bigskip
\noindent The computation of the $q_i$'s and $N$'s is the same as
before, and each vertex has $d=0$.

Now consider two elements in the Newton process:
$$(\sg_{(p_1,q_1,\mu_1)},\cdots,\sg_{(p_n,\tilde{m}_n,\mu_n)},\sg_{(p_{n+1},\tilde{m}_{n+1},\mu)};k),$$
$$(\sg_{(p'_1,q'_1,\mu'_1)},\cdots,\sg_{(p'_{n'},\tilde{m'}_{n'},\mu'_{n'})},\sg_{(p'_{n'+1},\tilde{m'}_{n'+1},\mu)};k').$$
%If for $i$ from $1$ to $n+1$ we have $p_i=p'_i, q_i=q'_i,\mu_i=\mu'_i$ then we have $k=k'$.
Suppose there is a first integer $j$ such that $p_i/q_i=p'_i/q'_i$
for $i\leq j$, $\mu_i=\mu'_i$ for $i<j$ and $\mu_{j}\neq \mu'_{j}$.
Then the Newton tree of the Newton process with these two elements
is as in Figure 16.

\bigskip
 \begin{figure}[ht]
 \begin{center}
\includegraphics{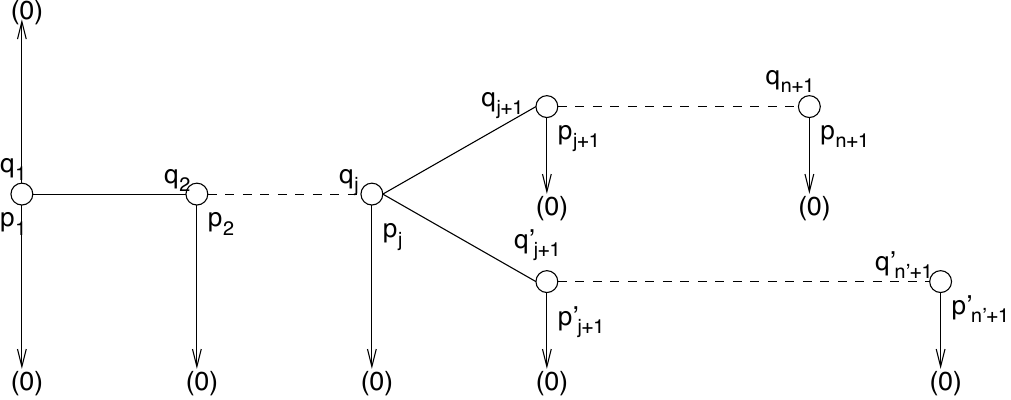}
\caption{}
 \end{center}
 \end{figure}
\bigskip

\noindent
Otherwise there is a first integer $j$ such that $p_i/q_i=p'_i/q'_i$
 and $\mu_i=\mu'_i$ for $i< j$ and $p_j/q_j>p'_j/q'_j$.
Then the Newton tree is as in Figure 17.

\bigskip
 \begin{figure}[ht]
 \begin{center}
\includegraphics{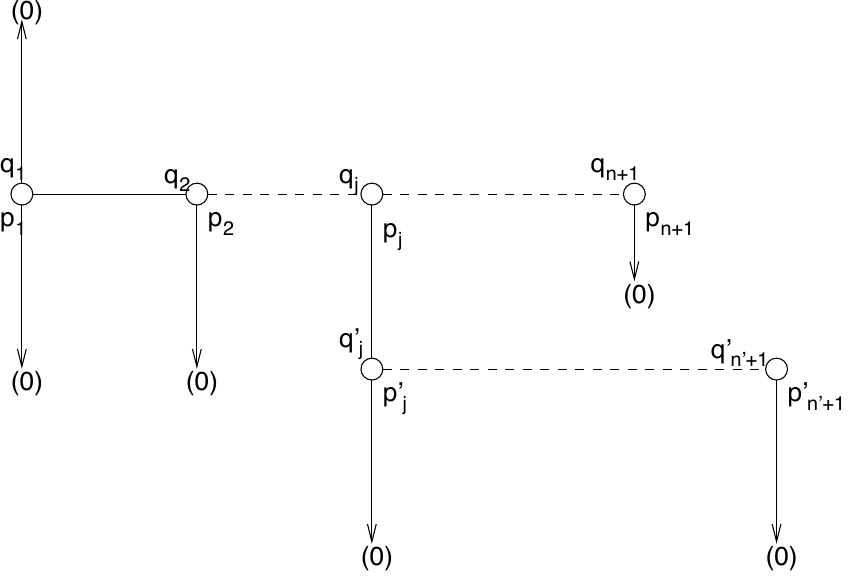}
\caption{}
 \end{center}
 \end{figure}
\bigskip

\noindent We can compute the decorations of the vertices using
Proposition \ref{decorationN}. In the case of branches the
construction is the same except that we have arrows at the final
vertices.

\begin{remark}\label{puiseux}
(1) The Newton process allows us to recover the Newton tree, but we
cannot write the Newton process from the Newton tree since there we
don't keep track of the $\mu_i$'s and the expressions $y+h(x)$.

(2) In the case of a principal ideal $(f)$ the Newton process  is
 similar to the Puiseux expansion of the branches of $f$, and the
 Newton tree to the
construction of splice diagrams from Puiseux expansions given by
Eisenbud and Neumann in \cite{EN}.
\end{remark}

 \bigskip
\noindent {\bf Example 5.} In this example we consider two ideals
with the same Newton tree, but different Newton process. Let
$$\Ic_1=(2x^4-x^2y^3+x^5, xy^5+x^2y^6,y^7+xy^6)$$ and
$$\Ic_2=(3x^4-x^2y^3,x^3y^2,y^7).$$
Their Newton tree is given in Figure 18.

  \begin{figure}[ht]
 \begin{center}
\includegraphics{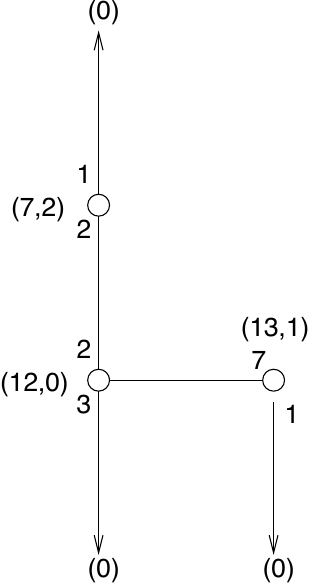}
\caption{}
 \end{center}
 \end{figure}

\noindent The Newton process of $\Ic_1$ is
 $$\{(\sg_{(3,2,2)},\sg_{(1,1,\mu)};1),(\sg_{(2,1,\mu)};2)\}$$
 and the Newton process of $\Ic_2$ is
 $$\{(\sg_{(3,2,3)},\sg_{(1,1,\mu)};1),(\sg_{(2,1,\mu)};2)\}.$$

\bigskip

\subsection{Newton process of the integral closure of an ideal}

In this paragraph we compare the Newton process of an ideal  and the
Newton process of its integral closure. Denote by $\overline{\Ic}$ the integral closure of an ideal $\Ic$.

The following result was proved in \cite{BFS} for ideals of finite codimension. The general case follows easily from that case.

\begin{lemma}\label{same Newton diagram}
  Let $\Ic$ be an ideal in $\bc[[x,y]]$.
 We have
 $$\Dg(\overline{\Ic})=\Dg(\Ic).$$
  \end{lemma}

\begin{prop}\label{integral closure}
 Let $\Ic$ be a non-trivial ideal in $\bc[[x,y]]$.
 The Newton process of the integral closure of $\Ic$ is the same as the Newton process of $\Ic$.
 \end{prop}

\begin{proof}
%We consider a Newton map $\sg $ followed by a change of variables $t$; denote $s=t\circ \sg$.
%Let $h\in \overline{\Ic}$. By definition there exist $a_i \in \Ic ^i$ such that
%$$h^k+a_1h^{k-1}+\cdots+a_{k-1}h+a_k=0 .$$
%Then
%$$s (h)^k+ s (a_1) s (h)^{k-1}+\cdots+ s (a_{k-1}) s(h) +s(a_k)=0, $$
%where  $s (a_i)\in (s (\Ic))^i$. This means that $s (h) \in \overline{s (\Ic)}$.
Let $\sg$ be a Newton map. Since it is a ring homomorphism, we have
$\sg (\overline{\Ic})\subset \overline{\sg (\Ic)}$.
Then we have
$$\Dg(\sg (\Ic))\subset \Dg(\sg (\overline{\Ic})) \subset \Dg(\overline{\sg\ (\Ic)})$$
and hence we conclude, by Lemma \ref{same Newton diagram}, that
$$\Dg(\sg (\Ic))= \Dg(\sg (\overline{\Ic})).$$

%The ideal $\sg(\Ic)=x^N\Ic_1$ where $\Ic_1$ has finite codimension.
%Since $\overline{\sg\ (\Ic)}=x^N\overline{\Ic_1}$ we have $\Delta
%(\Ic_1)=\Delta(\overline {\Ic_1})$. But the Newton diagram of
%$\sg(\Ic)=x^N\Ic_1$ is the translation by the vector $(N,0)$ of the
%Newton diagram of $\Ic_1$, which is the translation by $(N,0)$ of
%the Newton diagram of  $\overline{\Ic_1}$. Since
%$x^N\overline{\Ic_1}=\overline{\sg\ (\Ic)}$, then we conclude that
%$$\Dg(\sg (\Ic))= \Dg(\sg (\overline{\Ic})).$$

In order to show now that $\Ic$ and $\overline{\Ic}$
have the same Newton process, we may restrict to the case when $\Ic$ is of finite codimension. (Indeed, we have that  $\overline{(f)\Ic_1}=(f) \overline{\Ic_1}$ for $f\in \bc[[x,y]]$ and $\Ic_1$ an ideal in $\bc[[x,y]]$, see \cite{SH}.)

Let $S$ be a face of the Newton polygon of $\Ic$ (and $\overline{\Ic}$), with equation $p_S\ag+q_S\bg=N$, and

$$ \ti(\Ic,S)=x^{a_S}y^{b_S}F_{\Ic,S}(x^{q_S},y^{p_S})\big(k_1(x^{q_S},y^{p_S}),\cdots,k_s(x^{q_S},y^{p_S})\big)$$
or
$$
\ti(\overline{\Ic},S)=x^{a_S}y^{b_S}G_{\overline{\Ic},S}(x^{q_S},y^{p_S})\big(h_1(x^{q_S},y^{p_S}),\cdots,h_t(x^{q_S},y^{p_S})\big)$$
as in (\ref{initialideal1}) or (\ref{initialideal2}).
% Let $\mu$ be a root of $F_{\Ic,S}$ of multiplicity $\nu$, then the Newton polygon of $\sg_{(p_S,q_S,\mu)}(\Ic)$ has its first face with vertex $(N,\nu)$. Then $\sg (\overline{\Ic})$ also, and $\sg_{(p_S,q_S,\mu)}$ is an element of a sequence of the Newton process of $\sg (\overline{\Ic})$. Moreover $\mu$ is a root of
%$G_{\overline{\Ic},S}$ of multiplicity $\nu$.
If $F_{\Ic,S}$ is constant, then the face is dicritical and after a
Newton map $\sigma=\sg_{(p_S,q_S,\mu)}$ with $\mu$ generic we have a
Newton diagram of the form $(N,0)+\br^2$ (Figure 19).

  \begin{figure}[ht]
 \begin{center}
\includegraphics{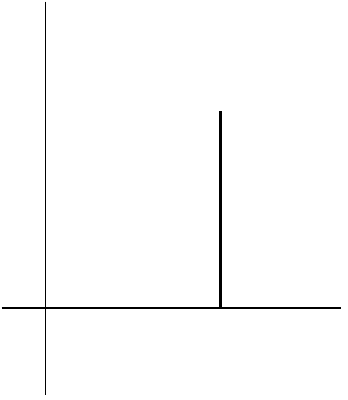}
\caption{}
 \end{center}
 \end{figure}

If $F_{\Ic,S}$ is not constant, let $\mu$ be one of its roots with
multiplicity $\nu$. After the Newton map
$\sigma=\sg_{(p_S,q_S,\mu)}$ the Newton polygon is either empty with
Newton diagram $(N,\nu)+\br ^2$ or is not empty with
 $(N,\nu)$ at the origin of the first face of the Newton polygon (Figure 20).

   \begin{figure}[ht]
 \begin{center}
\includegraphics{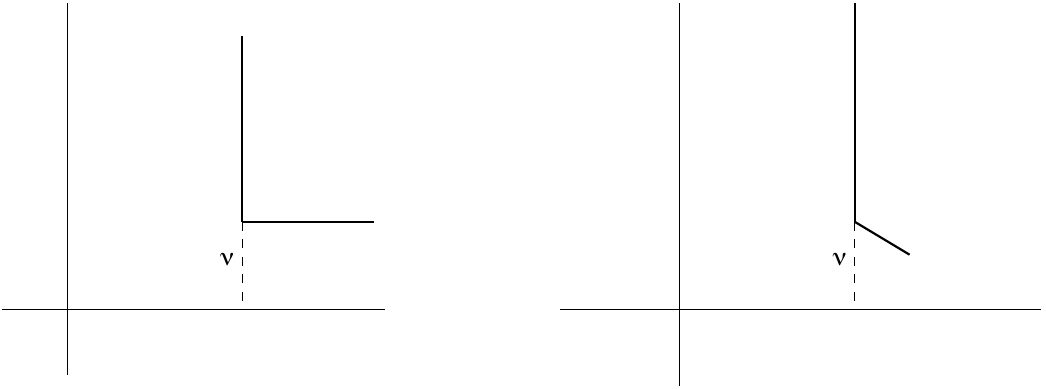}
\caption{}
 \end{center}
 \end{figure}

 Since $\sg(\Ic)$ and $\sg(\overline{\Ic})$ have the same Newton polygon for all $\sigma$,
  we must have $F_{\Ic,S}=G_{\overline{\Ic},S}$ and $\deg(k_i)=\deg(h_j)$.
Repeating the argument we conclude that $\Ic$ and $\overline{\Ic}$
have the same Newton process.
\end{proof}

\section{Computations of some invariants of the ideal using the Newton tree}

Recall that we denote by $\dc (\Ic)$ the set of dicritical vertices
of the Newton tree of the ideal $\Ic$ in $\bc[[x,y]]$.

\subsection{Vertices on Newton trees and valuations}

Let $\Ic$ be an ideal in $\bc[[x,y]]$ and $f\in \bc[[x,y]]$.  Let $v$ be a vertex on the Newton tree of $\Ic$. It corresponds to a face $S_v$ of a Newton polygon of some $\Ic_{\Sg}$, where $\Sg$ is a sequence of Newton maps. The face $S_v$ has equation $p\ag+\tilde{m}\bg=N_v$. We denote by
$\sg_v$ the Newton map $\sg_{(p,\tilde{m},\mu)}$ for a generic $\mu$.

Define $N_v(f)$ via
$$\sg_v(f_{\Sg})(x,y)=x^{N_v(f)}f_1(x,y)$$
with $f_1(x,y)\in \bc[[x,y]], f_1(0,y)\neq 0$. The map
$$\begin{matrix}
&\nu_v&:&\bc[[x,y]]&\longrightarrow&\bn\\
&&&f&\mapsto& N_v(f)\\
\end{matrix}$$
is a valuation on $\bc[[x,y]]$.

Using Proposition \ref{processofproduct}, we can identify the Newton
tree of $\Ic$ and the Newton tree of $f$ with subgraphs of the
Newton tree of $(f)\Ic$.

\begin{remark}\label{N is intersection mult}
The integer $N_v(f)$ is the intersection multiplicity of $f$ with an
irreducible element in $\bc[[x,y]]$ with Newton process
$\{(\Sg,\sg_v; y)\}$.

\noindent Indeed, as mentioned in Remark \ref{puiseux},
that Newton process corresponds to the Puiseux expansion of that
irreducible element.
\end{remark}

\begin{cor}\label{intersection multiplicity} Let $g\in\bc[[x,y]]$ and let $\Ic$ be a non-trivial ideal in $\bc[[x,y]]$ with generic curve $f$.
 The intersection multiplicity $(g,f)_0$ of $g$ with $f$ is equal to
 $$\sum_{v\in\mathcal{D}_{\Ic}}N_v(g)d_v.$$
 \end{cor}

%We consider the Newton tree of the ideal $(f)\Ic$. Because we know that the Newton polygon of $(f)\Ic$ is the Minkovski sum of the Newton polygons of $\Ic$ and $f$, we can identify the Newton tree of $\Ic$ and the Newton tree of $f$ with subgraphs of the Newton tree of $(f)\Ic$. Each vertex of the Newton tree of $\Ic$ can be seen as a vertex of the Newton tree of $(f)\Ic$.

%Note that if $f\in \Ic$ then there is a path between each  arrow representing  $f$ on the Newton tree of $(f)\Ic$ and a dicritical vertex of $\Ic$ which has no intersection with the Newton tree of $\Ic$. We will say for short that {\it $f$ separates from $\Ic$ after the dicriticals.}

The next proposition shows that we can compute the valuation of any element of $\bc[[x,y]]$ using the Newton tree.

\begin{prop}\label{formula for N}
We have for all vertices $v$ on the Newton tree of $\Ic$ that
$$ N_v(f)= \sum_{g\in \mathcal{F}_f} \rho_{v,g}m(g),$$
where $\mathcal{F}_f$ denotes the set of arrows representing $f$ on the Newton tree of $(f)\Ic$, and $m(g)$ the multiplicity of the arrow $g$ in the Newton tree of $f$.

\end{prop}

\begin{proof}
It is enough to prove the proposition in the case where $f$ is
irreducible; then there is only one arrow with multiplicity $1$. Let
$w$ be the vertex where $f$ is attached on the Newton tree of
$(f)\Ic$. We write $\mathcal{S}(v)=\{v_1,\cdots, v_i,
v_{i+1},\cdots, v\}$ and $\mathcal{S}(w)=\{v_1,\cdots, v_i,
w_{i+1},\cdots, w\}$, with $v_{i+1}\neq w_{i+1} $. We use induction
on $i$.

Assume first that $i=0$, meaning that $w_1\neq v_1$. Let $(q_1,p_1)$ be the
decorations near $v_1$ and let $(q'_1,p'_1)$  be the  decorations near
$w_1$. Assume that $p'_1q_1-p_1q'_1>0$ (Figure
21).

   \begin{figure}[ht]
 \begin{center}
\includegraphics{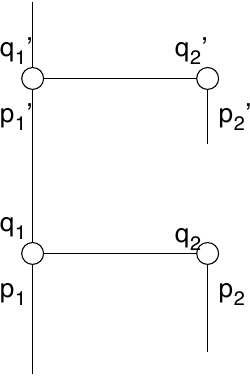}
\caption{}
 \end{center}
 \end{figure}

\noindent
We can write
$$f(x,y)=(y^{p'_1}-\mu'x^{q'_1})^K+\cdots,$$
where $\cdots$ means monomials with exponents above the face
$S_{w_1}$. By Lemma \ref{hauteur} we have $p'_1 K= \rho_{w_1w}/q'_1$ and consequently $K= \rho_{w_2w}/q'_2$. The first Newton map associated
with $v$ is $\sg_{(p_1,q_1,\mu)}$ for some $\mu\in \bc ^*$ and
$$f\big(\mu^{a_1}x^{p_1},x^{q_1}(y+\mu^{b_1})\big)=(x^{q_1p'_1}(y+\mu^{b_1})^{p'_1} - \mu'\mu^{a_1q'_1}x^{p_1q'_1})^K+\cdots=x^{p_1q_1'K}f_1(x,y),$$
where $f_1(0,y)\neq 0$.
Now we apply the composition of Newton maps which give rise to the
vertices $v_2,\cdots, v$ and we get
$$N_{v}(f)= p_1q_1'Kp_2p_3\cdots p = p_1q_1'K\rho_{v_2,v}/q_2=\rho_{v,w}.$$

Assume next that $i=1$ (Figure 22).

   \begin{figure}[ht]
 \begin{center}
\includegraphics{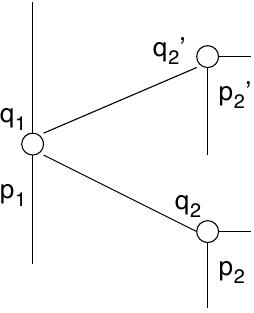}
\caption{}
 \end{center}
 \end{figure}

\noindent
We can write
$$f(x,y)=(y^{p_1} - \mu' x^{q_1})^K+\cdots,$$
where $\cdots$ means monomials with exponents above the face
$S_{v_1}$. The first Newton map associated with $v$ is
$\sg_{(p_1,q_1,\mu)}$ for some $\mu\in \bc ^*$ with $\mu \neq \mu'$
(because $i=1$). Hence
$$f\big(\mu^{a_1}x^{p_1},x^{q_1}(y+\mu^{b_1})\big)=(x^{q_1p_1}(y+\mu^{b_1})^{p_1} - \mu' \mu^{a_1q_1}x^{p_1q_1})^K+\cdots=x^{p_1q_1K}f_1(x,y),$$
where $f_1(0,y)\neq 0$,
and
$$N_{v}(f)=p_1q_1K\rho_{v_2,v}/q_2=\rho_{v,w}.$$

%Now we assume that for all $v,w$ such that  $\mathcal{S}(v)$ and $\mathcal{S}(w)$ have
%$i$ common elements for $i<i_0$ then $N_{v}(f)=\rho_{v,w}$ where $f$ is any arrow attached to $w$.
%We consider $v$ a vertex on the Newton tree of $\Ic$ and an $f\in \bc[[x,y]]$ attached to the Newton tree
%of $(f)\Ic$ by $w$.

Now we assume that $i=i_0>1$. We write again $f(x,y)=(y^{p_1} - \mu' x^{q_1})^K+\cdots$,
where $\cdots$ means monomials with exponents above the face
$S_{v_1}$.  We consider
$$\sg_{(p_1,q_1,\mu)}(f)(x,y)=x^{p_1q_1K}f_1(x,y).$$
Note that $\mathcal{S}(w)$ on the Newton tree of
$\sg_{(p_1,q_1,\mu)}((f)\Ic)$ has strictly less than $i_0$ vertices
in common with $\mathcal{S}(v)$. We  apply the induction hypothesis
to this Newton tree. We have (using Lemma \ref{hauteur})
$$N_{v}(f_1)=\tilde{q}_{i_0}p_{i_0}p_{i_0+1}\cdots pK/(p_2\cdots p_{i_0}),$$
where $(\tilde{q}_{i_0},p_{i_0}) $ are the decorations near
$v_{i_0}$ on the Newton tree of  $\sg_{(p_1,q_1,\mu)}((f)\Ic)$.

Then we have
$$
\begin{aligned}
N_v(f)&=p_1q_1Kp_2\cdots p_{i_0}\cdots
p+\tilde{q}_{i_0}p_{i_0}p_{i_0+1}\cdots pK/(p_2\cdots p_{i_0})\\
&=p_{i_0}Kp_{i_0+1}\cdots p \, (p_1q_1p_2^2\cdots
p_{i_0-1}^2p_{i_0}+\tilde{q}_{i_0})/(p_2\cdots p_{i_0}) \\
&=p_{i_0}q_{i_0}p_{i_0+1}\cdots p K/(p_2\cdots p_{i_0}),
\end{aligned}
$$
which, again by Lemma \ref{hauteur}, proves the proposition.
\end{proof}

\begin{cor}\label{N are same}
 Let $\Ic$ be a non-trivial ideal in $\bc[[x,y]]$ and $f$ a generic curve of $\Ic$. For all vertices $v$ of the Newton tree of $\Ic$
 we have
 $$N_v(f)=N_v.$$
 \end{cor}

\begin{proof}
Combine Proposition \ref{formula for N} and Proposition \ref{decorationN}.
\end{proof}

% \begin{cor}
% Assume that $\Ic$ has support $\{0\}$.
% If $f \in \Ic$ then $N_{v,f}\geq N_v$ for all $v\in \mathcal{D}_{\Ic}$.
% \end{cor}

% \begin{proof}
%If $f\in \Ic$ then there is a path between each  arrow representing  $f$ on the Newton tree of $(f)\Ic$ and a dicritical vertex of $\Ic$ which has no intersection with the Newton tree of $\Ic$. We will say for short that {\it $f$ separates from $\Ic$ after the dicriticals.}
%Then we use  the computation of $N_{v,f}$ and $N_v$ on the tree.
% \end{proof}

 \subsection{Multiplicity of an ideal}

\begin{definition}
Let $\Ic$ be a non-trivial ideal in $\bc[[x,y]]$ of finite codimension. We denote by $m (\Ic)$ the largest $k\in \bn$ such that $\Ic \subset \mc^k$.

\end{definition}

\begin{remark}
We have $m(\Ic)=m(f)$, where $m(f)$ is the multiplicity  (at the origin) of a generic curve $f$ of $\Ic$.
\end{remark}

\begin{definition}\label{rho}
Let $v$ be a vertex on a Newton tree, and $\cls (v)=\{v_i,\cdots, v_1,v\}$. We define
$$\rg (v)=\min \{ p_ip_{i-1} \cdots p_1 p, q_ip_{i-1}\cdots p_1p\} .$$
\end{definition}

\smallskip
\begin{remark}\label{originoftree}
 Note that $\rg(v)=\rg_{v_0,v}$, where $v_0$ is the vertex on the Newton tree with near decorations $(1,1)$. If the vertex $v_0$ does not appear on the Newton tree of $\Ic$ one considers $(x+y)\Ic$ and the vertex $v_0$ is a vertex of the Newton tree of this ideal.
\end{remark}

\smallskip
\begin{prop}\label{formula for m}
 Let $\Ic$ be a non-trivial ideal in $\bc[[x,y]]$ of finite codimension. We have
$$m(\Ic)=\sum_{v\in \dc(\Ic)}\rg (v) d_v .$$
\end{prop}

\begin{proof}
We use Proposition \ref{formula for N} and the previous remarks, and the fact that $m(f)=N_{v_0}(f)$.
\end{proof}

\subsection{Degree function of an ideal. Dicritical vertices and Rees valuations}

%\begin{prop}\label{gen-elements}
%Let $\Ic=(f_1,\cdots, f_r)$ be an ideal of $\bc[[x,y]]$ with support $\{0\}$. Then the Newton tree of $\Ic$ is the same as the Newton tree of $\Jc=(g_{\lbg},g_{\mu})$ for any sufficiently generic elements $g_{\lbg},g_{\mu}$ of $\Ic$.
%\end{prop}

%\begin{proof}
%The proof is an immediate consequence of Proposition 4.1.
%\end{proof}

%\begin{rem}
%Above the dicritical vertices of the Newton tree of the ideal $\Ic$ are precisely the vertices $v$ where the curves
%$g_{\lbg},g_{\mu}$ separate (that is, where two different sets of $d_v$ arrows are attached corresponding to both curves).
%\end{rem}

%Now we can compute the Hilbert-Samuel multiplicity of an ideal $\Ic$ with support $\{0\}$.

\begin{definition}
Let $(R,\mathcal{M})$ be a local noetherian ring of (Krull)
dimension $n$. If $\Ic$ is an $\mathcal{M}$-primary ideal in $R$, one defines its Hilbert-Samuel multiplicity
$e(\Ic)$ as
$$e(\Ic)=\lim_{k\to \infty}\frac{n}{k^n}\lambda(R/\Ic^k ),$$
where $\lambda$ is the length.
\end{definition}

\smallskip
Following Rees \cite{DV}\cite{Re}, we define the degree function $d_{\Ic}(f)$
of an element $f\in \mathcal{M}\setminus \{0\}$ with respect to the
ideal $\Ic$ as
$$d_{\Ic}(f)=e\left(\frac{\Ic+fR}{fR}\right).$$
To every prime divisor $\nu$ of $R$ one can assign a nonnegative
integer $d(\Ic,\nu)$, satisfying $d(\Ic,\nu)=0$ for all but a finite
number of $\nu$, such that
$$d_{\Ic}(f)=\sum_{\nu\in \mathcal{P}}d(\Ic,\nu)\nu(f)$$
for all $f\in \mathcal{M}\setminus \{0\}$, where $\mathcal{P}$
denotes the set of prime divisors of $R$. Moreover, Rees and Sharp \cite{RS} proved that the numbers $d(\Ic,\nu)$ are uniquely determined by
this condition, that is, if
$$\sum_{\nu\in \mathcal{P}}d(\Ic,\nu)\nu(f)=\sum_{\nu\in \mathcal{P}}d'(\Ic,\nu)\nu(f)$$
for all nonzero $f\in \mathcal{M}$, then for all $\nu \in \mathcal{P}$ we have $d(\Ic,\nu)=d'(\Ic,\nu)$.

The valuations $\nu$ such that $d(\Ic,\nu)\neq 0$ are called the {\it Rees valuations} of $\Ic$.

%In this paragraph we will study the relations between  the Rees
%valuations of $\Ic$  and the dicritical vertices of the Newton tree
%of $\Ic$ for an ideal in $\bc[[x,y]]$.

\begin{prop}\label{degree function}
Let $\Ic$ be a non-trivial ideal of finite codimension in $\bc[[x,y]]$ and
$f\in (x,y), f\neq 0$. Then
$$d_{\Ic}(f)=\sum_{v\in \mathcal{D}_{\Ic}}N_v(f) d_v.$$
\end{prop}

\begin{proof}
We have $d_{\Ic}(f)=e_1((f)\vert \Ic)$ \cite{DV}, where $e_1(\Ic \vert
\mathcal{J})$ is the mixed multiplicity of the ideals $\Ic$ and
$\mathcal{J}$. Rees \cite{Bi}\cite{Re} proved that for sufficiently generic
elements $g_{\Ic}$ and $g_{\mathcal{J}}$ of $\Ic$ and $\mathcal{J}$,
we have $e_1(\Ic \vert \mathcal{J})=e((g_{\Ic},g_{\mathcal{J}}))$.
Moreover $e((g_{\Ic},g_{\mathcal{J}}))=(g_{\Ic},g_{\mathcal{J}})_0$,
the intersection multiplicity of $g_{\Ic}$ and $g_{\mathcal{J}}$  in
$\bc[[x,y]]$. The result then follows from Corollary \ref{intersection multiplicity}.
\end{proof}

\begin{prop}\label{rees is dicritical}
The set of valuations $\nu_v, v\in \mathcal{D}_{\Ic},$ is the set of Rees valuations of $\Ic$, and
for each $v\in  \mathcal{D}_{\Ic}$ we have $d(\Ic, \nu)=d_ v$.
%Using the result of Rees and Sharp, we see that the composition of Newton maps and changes of variables that produce dicritical vertices are the Rees valuations of $\Ic$ and $d(\Ic, \nu)=d_ v$.
\end{prop}

\begin{proof}
Combine Proposition \ref{degree function} with the above mentioned
result of Rees and Sharp.
\end{proof}

Another (particular case of a) result of Rees and Sharp \cite[Corollary 5.3]{RS} is that, if $\Ic$ and $\mathcal{J}$
are two ideals in $\bc[[x,y]]$ of finite codimension, then the following statements
are equivalent:
\begin{enumerate}
\item $\overline {\Ic}=\overline{\mathcal{J}}$,
\item $d(\Ic,\nu)=d(\mathcal{J},\nu)$ for all $\nu \in \mathcal{P}$.
\end{enumerate}

%\begin{prop}
%The following statements are equivalent:
%\begin{enumerate}
%\item $f\in \overline{\Ic}$,
%\item $f$ separates from $\Ic$ after the dicriticals,
%\item $N_{v,f}\geq N_v$ for all $v\in \mathcal{D}_{\Ic}$.
%\end{enumerate}
%\end{prop}

% \begin{proof}
% The tree of $\overline{\Ic}$ is the same as the tree of $\Ic$. Hence if $f\in \overline{\Ic}$, then
% $f$ separates from $\Ic$ after the dicriticals, and then $N_{v,f}\geq N_v$ for all $v\in \mathcal{D}_{\Ic}$.
% Now we use Prop 10.5.2 of [SH] (7) to prove that $(3)$ implies $(1)$.
% \end{proof}

\smallskip

\begin{theorem}
Two non-trivial ideals in $\bc[[x,y]]$  have the same Newton process if and only
if  they have the same integral closure.
\end{theorem}

\begin{proof}
Let $\Ic_1$ and $\Ic_2$ be two ideals in $\bc[[x,y]]$. Assume they
have the same integral closure. Since (by Proposition \ref{integral
closure})  the Newton process of the integral closure is equal to
the Newton process of the ideal, they have the same Newton process.

Assume they have the same Newton process. We can write
$\Ic_1=(f)\Ic'_1$ with $\Ic'_1$ of finite codimension, and
$\Ic_2=(f)\Ic'_2$ with $\Ic'_2$ of finite codimension, and then
$\Ic'_1$ and $\Ic'_2$ have the same Newton process. Proposition
\ref{rees is dicritical} and the second result of Rees and Sharp
imply that $\overline{\Ic'_1}=\overline{\Ic'_2}$ and consequently
$\overline{\Ic_1}=\overline{\Ic_2}$.
\end{proof}

\bigskip
\noindent {\bf Example 6.} Consider the two ideals
$$\Ic_1=(x^2,xy^4,y^5) \qquad\text{and}\qquad \Ic_2=(x^2,xy^3,y^5).$$
They have the same Newton process $\{(\sg_{(5,2,\mu)};1)\}$,
therefore they have the same integral closure.

\begin{remark}
Assume that  the ideal $\Ic$ in $\bc[[x,y]]$ is non degenerate. Let
$\Ic_0$ be the monomial ideal generated by the elements
$x^{\ag}y^{\bg}$, where $(\ag,\bg)$ are all the vertices of the
Newton polygon of $\Ic$. Then $\Ic$ and $\Ic_0$ have the same Newton
process, hence the same integral closure. But the integral closure
of a monomial ideal is a monomial ideal. Thus if $\Ic$ is non
degenerate, its integral closure is a monomial ideal. This result
has already been proved in \cite{S}. Reciprocally, if the integral
closure of $\Ic$ is a monomial ideal, then it is non degenerate and $\Ic$
itself is non degenerate.
\end{remark}

\subsection{Factorization of the integral closure of an ideal}

Recall the following result of Zariski.

\begin{theorem}
Every non-zero integrally closed ideal $\Ic$ in $\bc[[x,y]]$ can be
written uniquely (except for ordering of the factors) as
$$\Ic=f_1^{l_1}\cdots f_m^{l_m}\Ic_1^{k_1}\cdots \Ic_n^{k_n},$$
where $\Ic_1,\cdots,\Ic_n$ are simple integrally closed ideals of
finite codimension, $f_1,\cdots,f_m$ are irreducible elements in
$\bc[[x,y]]$ and $l_1,\cdots,l_m,k_1,\cdots,k_n$ are positive
integers.
\end{theorem}

Such a decomposition can be obtained using the Newton process. The
ideals $\Ic_j$ are the integrally closed ideals with Newton process
$\{(\Sg;1)\}$ such that  $(\Sg;k_j)$ belongs to the Newton process
of $\Ic$. The irreducible elements $f_i$ are the irreducible
elements with Newton process $\{\big(\Sg; (y+h_i(x))\big)\}$ such that
$\big(\Sg; (y+h_i(x))^{l_i}\big)$ belongs to the Newton process of $\Ic$.

\bigskip
\noindent {\bf Example 2 {\rm (continued)}.} The Newton process of
$\Ic$ is
$$\{(\sg_{(1,1,\mu)};3),(\sg_{(1,3,\mu)};1)\}.$$
Then
$$\overline{\Ic}=\Ic_1^3\Ic_2,$$
where $\Ic_1$ is the integrally closed ideal with  Newton process
$\{(\sg_{(1,1,\mu)};1)\}$, that is $(x,y)$, and $\Ic_2$ is the
integrally closed ideal with Newton process
$\{(\sg_{(1,3,\mu)};1)\}$, that is $(x^3,y)$. Hence
$$\overline{\Ic}=(x,y)^3(x^3,y).$$

\bigskip
\noindent {\bf Example 3 {\rm (continued)}.} The Newton process of $\Ic$ is
$$\{ (\sg_{(3,2,1)}, \sg_{(1,4,\mu)};1), $$
$$\begin{aligned} (\sg_{(3,2,-1)},\sg_{(2,1,-1/9)}, \sg_{(1,1,-1/54)},\sg_{(1,1,-55/17496)},&\sg_{(1,1,-10/19683)}, \\ &\sg_{(1,1,-7/93312)},\sg_{(1,1,-91/9565938)},\sg_{(1,1,\mu)};1),
\end{aligned}$$
$$(\sg_{(1,2,\mu)};1),(\sg_{(1,4,-1)}, \sg_{(1,2,\mu)};1)\}.$$
Then $\overline{\Ic}$ is the product of four simple integrally closed ideals:
$$\overline{\Ic}=\Ic_1\Ic_2\Ic_3\Ic_4 .$$
To find generators of these ideals, we use the following result in  \cite [Theorem 4.3]{RS}:
$x\in \overline{\Ic} $ if and only if $v(x)\geq v(\Ic)$ for all Rees valuations $v$ of $\overline{\Ic}$.

\smallskip
Let $\Sg_1=\{ (\sg_{(3,2,1)}, \sg_{(1,4,\mu)};1)\}$.
We claim that the simple integrally closed ideal with Newton process $\Sg_1$ is $$\Ic_1=(x^2-y^3,x^2y^2,x^3y).$$
It is easy to verify that this ideal has indeed $\Sg_1$ as Newton process; we must show that it is integrally closed.

We know that $X\in \overline{\Ic_1}$ if and only if
$v(X)\geq v(\Ic_1)=10$, where $v$ is the unique Rees valuation involved. Hence all monomials $x^{\ag}y^{\bg}$ with $3\ag+2\bg \geq 10$ belong to $\overline{\Ic_1}$. Moreover, one can compute that $X=\sum_{3\ag+2\bg<10}c_{\ag,\bg}x^{\ag}y^{\bg}$ belongs to $\overline{\Ic_1}$ if and only if $X\in (x^2-y^3)$. Finally one verifies immediately that $x^2-y^3$, $x^2y^2$ and $x^3y$ generate the other monomials $x^{\ag}y^{\bg}$ with $3\ag+2\bg \geq 10$, being $x^4$, $xy^4$ and $y^5$.

\smallskip
Let $$\Sg_2=\{(\sg_{(3,2,-1)},\sg_{(2,1,-1/9)}, \sg_{(1,1,-1/54)},$$
$$\sg_{(1,1,-55/17496)},\sg_{(1,1,-10/19683)},\sg_{(1,1,-7/93312)},\sg_{(1,1,-91/9565938)},\sg_{(1,1,\mu)};1)\}.$$
We claim that
$$\Ic_2=((x^2+y^3)^2+xy^5,x(x^2+y^3)^2, x^4y^2,y^8,x^5y)$$
is the simple integrally closed ideal with Newton process $\Sg_2$.
It is straightforward  to verify that $\Ic_2$ has $\Sg_2$ as Newton process; we must show that it is integrally closed.

Now $X\in \overline{\Ic_2}$ if and only if
$v(X)\geq v(\Ic_2)=32$, where $v$ is the unique Rees valuation involved.
Hence all monomials $x^{\ag}y^{\bg}$ with $3\ag+2\bg\geq 16$ are in $\overline{\Ic_2}$. Also, a more tedious calculation shows that
$X=\sum_{3\ag+2\bg<16}c_{\ag,\bg}x^{\ag}y^{\bg}$ belongs to $\overline{\Ic_2}$ if and only if $X\in ((x^2+y^3)^2+xy^5,x(x^2+y^3)^2)$.
 Since $(x^2+y^3)^2+xy^5$ and $x(x^2+y^3)^2$ belong to $\Ic_2$, also $x^2y^5 \in \Ic_2
 $. Next $x^2y^5, x^4y^2, y^8,y^2((x^2+y^3)^2+xy^5)\in \Ic_2$ implies that $xy^7\in \Ic_2$. Further $xy^7,x^5y,x^2y^6,xy((x^2+y^3)^2+xy^5)\in\Ic_2$ implies that $x^3y^4\in\Ic_2$. Finally also $x^6$ belongs to $\Ic_2$. As a conclusion all monomials with $3\ag+2\bg\geq 16$ are in $\Ic_2$.

\smallskip
Let
$\Sg_3=\{(\sg_{(1,2,\mu)};1)\}$.
Then $\Ic_3=(x^2,y)$.

Let $\Sg_4=\{(\sg_{(1,4,-1)}, \sg_{(1,2,\mu)};1)\}$.
Then $\Ic_4=(x^4+y,x^2y)$.

\bigskip
\noindent {\bf Example 5 {\rm (continued)}.}
 We have
$$\Ic_1=\mathcal{J}_1\mathcal{J}_2^2 \quad\text{and} \quad \Ic_2=\mathcal{J}'_1\mathcal{J}_2^2,$$
where $\mathcal{J}_1$ is the integrally closed ideal with Newton
process $\{(\sg_{(3,2,2)},\sg_{(1,1,\mu)};1)\}$, that is
$(2x^2-y^3,xy^2,y^4)$ \cite[ex. 1.3.3]{SH},  $\mathcal{J}'_1$ is the
integrally closed ideal with Newton process
$\{(\sg_{(3,2,3)},\sg_{(1,1,\mu)};1)\}$, that is
$(3x^2-y^3,xy^2,y^4)$, and $\mathcal{J}_2$ is the integrally closed
ideal with Newton process $\{(\sg_{(2,1,\mu)};1)\}$, that is
$(x,y^2)$.

\smallskip
\subsection{Hilbert-Samuel multiplicity of an ideal}

%For any sufficiently generic elements $g_{\lbg},g_{\mu}$  in $\Ic$, we have (see [SH])
%$$e(\Ic)=e((g_{\lbg},g_{\mu})) .$$
%Now if $\Jc=(g_{\lbg},g_{\mu})$, then
%$$e(\Jc)=\dim_{\bc} \bc[[x,y]]/(g_{\lbg},g_{\mu}) ,$$
%which is the intersection multiplicity of $g_{\lbg}$ and $g_{\mu}$ at the origin.

%????REFERENCE???

\smallskip
\begin{theorem}\label{HSfirst}
Let $\Ic$ be a non-trivial ideal in $\bc[[x,y]]$ of finite codimension. Then
$$e(\Ic)=\sum_{v \in \dc (\Ic)} N_v d_v .$$
\end{theorem}

\begin{proof}
This is the formula $e(\Ic)=\sum_{v \in \mathcal{R}(\Ic)}v(\Ic)d(\Ic,v)$ \cite[Theorem 4.3]{RS}, where $\mathcal{R}(\Ic)$ denotes the set of Rees valuations of $\Ic$.

%???? REFERENCE (OR PROOF) ????
\end{proof}

%We can compute the Hilbert-Samuel multiplicity in the examples. We have for instance $e(\Ic_1)=18$ and $e(\Ic_3)=12$.

\bigskip
\noindent {\bf Example 3 {\rm (continued)}.} Looking at the Newton
tree of the ideal $\Ic$  we calculate $e(\Ic)= \sum_{v \in \dc
(\Ic)} N_v d_v = 10+26+52+14=102$.

\bigskip
When the non-trivial ideal $\Ic$ is not of finite codimension,  one introduces
$$j(\Ic)=\lim_{n\to \infty}\frac{1}{n}\lambda\left(\Gamma_m(\Ic^n/\Ic^{n+1})\right)$$
where $\lambda (C)$ is the length of the $\bc[[x,y]]$-module $C$ and
$\Gamma_m(D)$ is the zero-th local cohomology functor applied to
$D$. When $\Ic$ has finite codimension, $j(\Ic)$ is the usual
multiplicity $e(\Ic)$.

It is proved in \cite{KV} that, if $\Ic=(f)\Ic_1$ with $\Ic_1$ of finite
codimension, then
$$j(\Ic)=e(\Ic_1)+d_{\Ic_1}(f).$$
Then by Proposition \ref{degree function} the previous theorem is
still valid in this more general setting.

\begin{theorem}
Let $\Ic=(f)\Ic_1$ be a non-trivial ideal in $\bc[[x,y]]$, with $\Ic_1$ of finite codimension. Then
$$j(\Ic)=\sum_{v \in \dc (\Ic_1)} N_v d_v .$$
\end{theorem}

\bigskip
We develop another computation of the Hilbert-Samuel multiplicity of an ideal of finite codimension using regions in the plane limited by
Newton polygons.

If an ideal $\Ic$ in $\bc[[x,y]]$ has a Newton polygon such that
$S_1$ has its origin on the line $\{x=N\}$ and $S_m$ its extremity
on the line $\{y=0\}$, we denote by $m(\Ic)$ the area of the region
limited by the lines $\{x=N\}$, $\{y=0\}$ and the Newton polygon (Figure 23).
%  Let us denote by $m(\Ic)$ the area of the region limited by the Newton polygon of $\Ic$.

       \begin{figure}[ht]
 \begin{center}
 \includegraphics{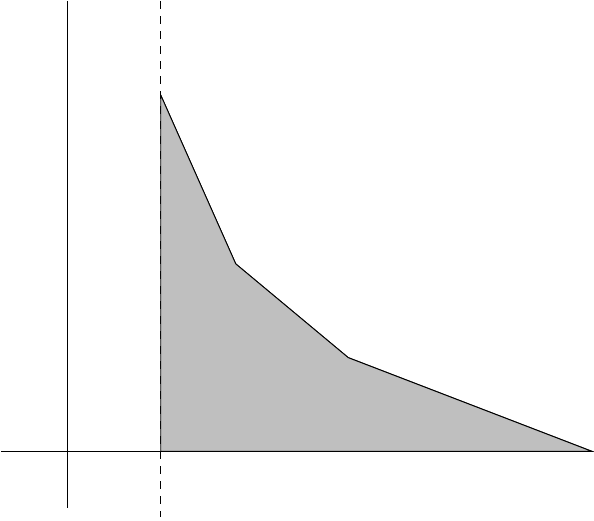}
 \caption{}
 \end{center}
\end{figure}

Let $\Ic$ be an ideal of finite codimension in $\bc[[x,y]]$. It is
proven in \cite{BFS}\cite{Ko} that
$$e(\Ic)=2m(\Ic)$$
if and only if $\Ic$ is non degenerate, and that $e(\Ic)>2m(\Ic)$
otherwise.

\bigskip
\noindent {\bf Example 2 {\rm (continued)}.} Looking at the Newton
polygon of the ideal $\Ic$  we see that $e(\Ic)=2m(\Ic)= 18$. This
is confirmed by Theorem \ref{HSfirst}, saying that $e(\Ic)=12+6=18$.

\bigskip
We can compute $e(\Ic)$ in general, using the area of the regions associated with the successive Newton polygons that appear in the Newton process.
If $\Sg_i=(\sg_1,\cdots, \sg_i)$ is a sequence of Newton maps, then $\Ic_{\Sg_i}$ has a Newton polygon as before and $m(\Ic_{\Sg_i})$ is well defined.

\begin{theorem}\label{HSmult} Let $\Ic$ be a non-trivial ideal in $\bc[[x,y]]$ of finite codimension  and of depth $d$. Then
$$e(\Ic)=2\, \left(m(\Ic)+\sum_{i=2} ^{d} \sum_{\Sg_i}m(\Ic_{\Sg_i})\right) ,$$
the second summation being taken over all possible sequences of Newton maps of length $i$.
\end{theorem}

%If we come back to Example 3, we have $m(\Ic _3)=11$ and, after the Newton map, $m(\Ic_{3,\sg})=1$, thus $e(\Ic)=12$.

\noindent
In order to prove the theorem,
we consider first two lemma's.
%We define a partial order on the set of vertices and arrows of a Newton tree. The origin is the arrow on the top of the Newton tree. {\it We say that $v<v'$ if and only if $v$ is on the path between the origin and $v'$.}      WHY IS THIS NEEDED ????????

%Consider the Newton tree of a generic curve of an ideal. Let $v$ be a vertex on this Newton tree. {\it We define $\delta_v$ such that $\delta _v +1$ is the number of points with integer coordinates on the corresponding face $S$.} In other terms, if we write
%$$\ti (\Ic,S)=(h)(\tilde{h_1},\cdots, \tilde{h_s})$$
%where  $h(x,y)=k(x^q,y^p)$, $\tilde{h_i}(x,y)=k_i(x^q,y^p)$ and $\gcd_i k_i=1$, then $\delta _v$ is equal to the sum of the degree of $k$ and the degree of $k_i$.  Recall that for any vertex we defined $d_v$ to be the degree of $k_i$, and that a dicritical vertex is a vertex such that $d_v \geq 1$.

\begin{lemma}\label{lemma1}

Let $\Delta$ be a Newton diagram such that $\br_+^2 \setminus
\Delta$ is bounded. Let $m(\Delta)$ be the area of $\br_+^2
\setminus \Delta$. For each face $S$ of $\mathcal{N}(\Delta)$ with
equation $p_S\ag+q_S\bg=N_S$, denote by $\delta_S$ the number of
points with integral coordinates on $S$. Then
$$2m(\Delta)=\sum_{S\subset \mathcal{N}(\Delta)}N_S(\delta_S-1).$$
%Let $\Ic$ be an ideal in $\bc[[x,y]]$ with support $\{0\}$. Then
%$$2m(\Ic)=\sum_v N_v\delta_v$$
%the sum being taken over the set of vertices of the graph associated to the Newton diagram of $\Ic$.
\end{lemma}

\begin{proof}
Let $S$ be  a face of   $\mathcal{N}(\Delta)$. Denote by
$A_1=(x_1,y_1)$ and $A_2=(x_2,y_2)$ the origin and end of $S$,
respectively. Denote by $m(S)$ the area of the region limited by the
triangle $(A_1,0,A_2)$. Then
$$2m(S)=x_2y_1-x_1y_2=(x_2-x_1)y_1+x_1(y_1-y_2) .$$
Since
$x_2-x_1=q_S(\delta_S-1)$ and $y_1-y_2=p_S(\delta_S-1)$, we conclude that
$$2m(S)=(q_Sy_1+p_Sx_1)(\delta_S-1)=N_S(\delta_S-1) .$$
Obviously $m(\Delta)$ is the sum of these areas $m(S)$.
\end{proof}

\bigskip

If in Lemma \ref{lemma1} the depth of $\Ic$ is one, then all faces $S$ are dicritical and $\delta_S-1=d_S$.

\begin{lemma}\label{lemma2}
Let $\Sg$ be a sequence of Newton maps and $\Ic _{\Sg}$ the
corresponding ideal. Assume that its Newton polygon has its origin
on the line $x=N$. For each face $S$ of $\mathcal{N}(\Ic_{\Sg})$
with equation $p_S\ag+q_S\bg=N_S$, denote by $\delta_S$ the number
of points with integral coordinates on $S$. Then
$$2 m(\Ic_{\Sg})= \sum_{S\subset\mathcal{N}(\Ic_{\Sg})} (N_S-Np_S)(\delta_S-1).$$
\end{lemma}

\begin{proof}
Let $S$ be a face of the Newton polygon of $\Ic_{\Sg}$ with origin
$A_1=(x_1,y_1)$ and end $A_2=(x_2,y_2)$. Denote by $m(S)$ the area
of the region limited by the triangle $(A_1,(N,0),A_2)$.  Then
$$2m(S)=(x_2-N)y_1-(x_1-N)y_2=(x_2-x_1)y_1+x_1(y_1-y_2)-N(y_1-y_2) .$$
Since $x_2-x_1=q_S(\delta_S-1)$ and $y_1-y_2=p_S(\delta_S-1)$, we see that
$$2m(S_v)=(q_Sy_1+p_Sx_1-Np_S)(\delta_S-1)=(N_S-Np_S)(\delta_S-1) ,$$
and  $m(\Ic_{\Sg})$ is the sum of the areas $m(S)$.
\end{proof}

Now we can prove Theorem \ref{HSmult}.
\begin{proof}
We have by Lemma \ref{lemma1} that
$$2m(\Ic)=\sum_v N_v(\delta_v-1) ,$$
where $v$ runs over the vertices of the graph associated with
$\nc(\Ic)$. Here each vertex $v$ corresponds to a face $S_v$ whose
supporting line has equation $p_v\ag+q_v\bg=N_v$, and we set
$\delta_v=\delta_{S_v}$.

Let $\sg_v=\sg_{(p_v,q_v,\mu_i)}$ be a Newton map associated with
$S_v$, where $\mu_i$ is a root of the face polynomial $F_{\Ic,S_v}$.
We have by Lemma \ref{lemma2} that
$$2m(\Ic_{\sg_v})= \sum_{v_i} (N_{v_i}-N_vp_{v_i})(\delta_{v_i}-1) ,$$
where the sum is taken over the vertices $v_i$ of the graph
associated with the Newton diagram of $\Ic_{\sg}$. We have
$$\delta_v-1=d_v+\deg F_{\Ic,S_v}$$
and
$$\deg F_{\Ic,S_v}=\sum_{v'}p_{v'}(\delta_{v'}-1),$$
where now the sum is over all vertices $v'$ of the graphs associated
with the Newton diagrams of the $\Ic_{\sg}$ for {\it all} roots of
$F_{\Ic,S_v}$. For the last equality we use that $\Ic$ is of finite codimension. Then
$$\sum_{v'}p_{v'}(\delta_{v'}-1)=\delta_v-1-d_v$$ and hence
$$2m(\Ic)+\sum_v \sum_\sigma 2m(\Ic_{\sg_v})= \sum_v N_vd_v+ \sum_v  \sum_{v'} N_{v'}(\delta_{v'}-1) .$$
Continuing the same procedure we arrive at the formula
$$2\left(m(\Ic)+\sum_{i=2} ^{d} \sum_{\Sg_i}m(\Ic_{\Sg_i})\right)=\sum_w N_w d_w ,$$
where, on the right hand side, $w$ runs over all vertices of the Newton tree of $\Ic$.
Since $d_w \neq 0$ if and only if $w$ is dicritical, we obtain the stated formula for $e(\Ic)$.
\end{proof}

\bigskip
\noindent {\bf Example 3 {\rm (continued)}.} Looking at the Newton
polygons arising in the Newton process for the ideal $\Ic$ we
calculate $e(\Ic)= 88+4+2+1+1+1+1+1+1+2=102$.

\bigskip

\subsection{{\L}ojasiewicz exponent}

\begin{definition}
Let $\Ic=(f_1,\cdots, f_r)$ be an ideal in $\bc[[x,y]]$ of finite codimension.
The {\L}ojasiewicz exponent of $\Ic$, denoted by $\cl _0(\Ic)$, is the infimum of $\ag>0$ such that there exists
an open neighbourhood $U$ of $0$ in $\bc^2$ and a constant $c>0$ such that
$$\|x\|^{\ag}\leq c \sup_{1\leq i \leq r} \vert f_i(x)\vert $$
for all $x\in U$.
\end{definition}

By \cite{LT}, we have that
$$\cl_0(\Ic)=\min\{r/s \mid (x,y)^r \subseteq \overline{\Ic^s}\} ,$$
 and
$$\cl _0(\Ic)= \sup _{\Gg \in \mathcal{T}}\frac{\inf _{h\in \Ic} \text{ord}(h\circ \Gg)}{\text{ord}(\Gg)} ,$$
where $\mathcal{T}$ denotes the set of analytic maps $(\bc,0)\to (\bc^2,0)$.
The number on the right hand side of the previous equality is called the {\it order of contact of $\Ic$}
in \cite{MN}. They proved that, if
$$f_a=\prod _b f_{a,b}^{n_{a,b}}$$
is the decomposition of $f_a, 1\leq a\leq r,$ into irreducible components in $\bc[[x,y]]$,
then
\begin{equation}\label{lojasiewicz}
\cl_0(\Ic)=\max_{a,b}\frac {\min_i\text{ord}(f_i\circ \Gg_{a,b})}{\text{ord}(\Gg_{a,b})} ,
\end{equation}
where $\Gg_{a,b}$ is a parametrization of $\{f_{a,b}=0\}$.

\begin{prop}
Let $\Ic$ be an ideal in $\bc[[x,y]]$ of finite codimension. Then
$$\cl_0(\Ic)=\max_{v\in \dc(\Ic)} \frac{N_v}{\rho(v)} ,$$
where $\rho(v)$ is as in Definition \ref{rho}.
\end{prop}

\begin{proof}
We note that $\cl_0(\Ic)=\cl_0(\overline{\Ic})$. Let $g_{\lbg},g_{\mu}$ be two sufficiently generic
elements of $\Ic$ such that $e(\Ic)=e((g_{\lbg},g_{\mu}))$. Then $\overline{\Ic}=\overline{(g_{\lbg},g_{\mu})}$ and
$$\cl_0(\Ic)=\cl_0(\overline{\Ic})=\cl_0(\overline{(g_{\lbg},g_{\mu})} )=\cl_0((g_{\lbg},g_{\mu}) ) .$$
Now we can apply (\ref{lojasiewicz}) and previous results on multiplicity and intersection multiplicity, more precisely Remark \ref{N is intersection mult}, Corollary \ref{N are same} and Proposition \ref{formula for m}.
\end{proof}

\bigskip
\noindent {\bf Example 7.} We consider an example in \cite{MN}. Let
$$f_s=sx+x^3+y^8, \quad g=x^2-y^{101},$$
 and $\Ic_s=(f_s,g)$.
If $s\neq 0$, the Newton tree of $\Ic_s$ is the tree on the left in
Figure 24, and for $s=0$ it is the one on the right.
 Then we can compute that $\cl_0(\Ic_s)=16$ if $s\neq 0$, and that $\cl_0(\Ic_0)=8$.
       \begin{figure}[ht]
 \begin{center}
 \includegraphics{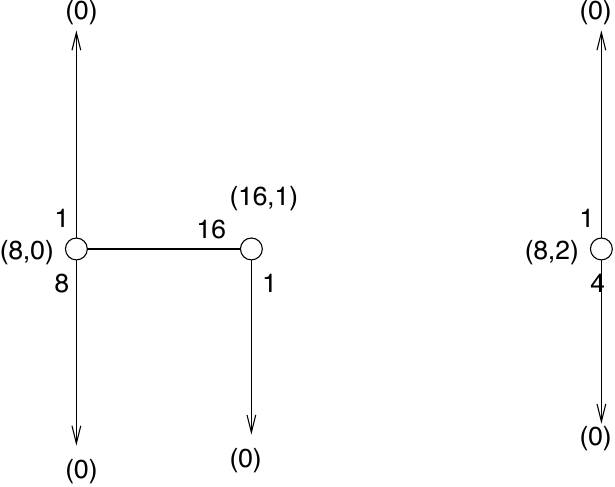}
 \caption{}
 \end{center}
\end{figure}

\begin{remark}
We note that in the previous example we have
$e(\Ic_s)=e(\Ic_0)=16$. It is proven in \cite{P} that in an $e$-constant
family of ideals, the {\L}ojasiewicz exponent is lower semi-continuous.
It is not the case in general.
\end{remark}

\bigskip
\section{Geometric version}

\bigskip
Our Newton algorithm in section 2 was developed from an algebraic point of view; we did not include coordinate changes.
This is in particular useful in studying when two ideals have the same integral closure.

From geometric point of view however, our Newton tree/process is sometimes not minimal.
Consider for example the principal ideal $\Ic$ in $\bc[[x,y]]$ generated by $(y+x^2+x^3)^2+x^7$.
Its Newton algorithm consists of three Newton maps $\sigma_1$, $\sigma_2$ and $\sigma_3$, given by
$$
\begin{cases} x=x_1    \\  y=x_1^2(y_1-1)
\end{cases}\qquad\begin{cases} x_1=x_2   \\  y_1=x_2(y_2-1)    \end{cases}
 \text{and}\qquad \begin{cases} x_2=x_3^2   \\  y_2=x_3(y_3-1) ,
\end{cases}
$$
respectively, resulting in the ideal $(x_3^{14} y_3)$.

If on the other hand we would perform first the change of coordinates $(x,y)\to(x,y-x^2-x^3)$ in $\bc[[x,y]]$, the given ideal would be
 generated by $y^2+x^7$. And then the Newton map $\sigma$ given by $x=x_1^2, y=x_1^7(y_1-1)$ immediately leads to the ideal $(x_1^{14} y_1)$.
We could say that the \lq algebraic depth\rq\ of $\Ic$ is $3$ and its \lq geometric depth\rq\ is $1$.

\medskip
In a sequel to this paper, we will develop a geometric version, using Newton maps and appropriate coordinate changes, yielding \lq minimal\rq\ Newton trees with geometric interpretation. In particular their vertices will correspond to exceptional components of the so-called relative log canonical model of the blow-up of $\operatorname{Spec} \bc[[x,y]]$ in $\Ic$.
In Example 3 for instance, this \lq minimal\rq\ Newton tree has only six vertices: the four dicritical ones and those with decorations $N=22$ and $N=46$.

\enddocument

\section{Computation of the motivic zeta function of an ideal}

In this section we express the motivic zeta function of an ideal in terms of the Newton tree. The method is the method used in [ACLM] for quasi-ordinary singularities. We can retrieve the results of [VV] for the topological zeta function of an ideal.

Let $\Gc$ be the Grothendieck ring of the algebraic varieties over $\bc$. It is generated by the symbols $[S]$ where $S$ is an algebraic variety over $\bc$, with the relations
$$[S]=[S']$$
when $S$ and $S'$ are isomorphic,
$$[S]=[S\setminus S']+[S']$$
when $S'$ is closed in $S$, and
$$[S\times S']=[S][S'].$$

Let $\bl=[\ba_{\bc}^1]$ and $\Gc_{\text{loc}}=\Gc[\bl ^{-1}]$. We denote by $\Gc[T]_{\text{loc}}$ the subring of $\Gc_{\text{loc}}[[T]]$ generated by $\Gc_{\text{loc}}[T]$ and the series
$$(1-\bl^{-a}T^b)^{-1}$$
with $a\in \bz,b\in \bn \setminus \{0\}$.

If $X$ is an algebraic variety over $\bc$, we denote by $\cl(X)$ the set of arcs on $X$.  If $R$ is a $\bc$-algebra
$$\cl_n(X)\sim\text{Mor}(\text{Spec} R[t]/t^{n+1}R[t], X)$$
$$\cl(X)=\lim_{\longleftarrow} \cl _n(X)$$
Denote by
$$\pi _n: \cl (X) \longrightarrow \cl_n(X)$$
If $A$ is a semialgebraic set of $\cl _0(\bc ^2)$, we say that it is {\it stable at order} $n$, if $\forall m\geq n$,
$$[\pi_m(A)]=[\pi _n(A)]\bl ^{2(m-n)}$$
Then we define $$\mu(A)=[\pi_n(A)]\bl ^{-2n}.$$
If $A$ is not stable, $\lim_{n\to \infty}[\pi_n(A)]\bl^{-2n}$ exists and we define
$$\mu(A)=\lim_{n\to \infty}[\pi_n(A)]\bl^{-2n}$$

Let $\Ic$ be an ideal in $\bc[[x,y]]$. We write $\Ic=(x^N)\Ic'$. Let $\og$ be a differential,
$$\og=x^{\nu-1}dx \wedge dy$$
We assume that if $N=0$, then $\nu=1$.

Let $\phi\in \cl _0(\bc^2)$. Define
$$\text{ord}_{\phi } \Ic=\min \{\text{ord}(f \circ \phi ), f \in \Ic\}$$

Let $$V_{n,m}=\{\phi\in \cl_0(\bc^2)\vert \text{ord}_\phi \Ic=n, \text{ord}(\og \circ \phi )=m\}$$

Define
$$\zg_{DL}(\Ic,\og)(T)=\sum_{n\geq 1}(\sum_{m\geq 1}\mu(V_{n,m})\bl^{-m})T^n$$
 The hypothesis ensures that the sum over $m$ is finite.

 First we consider the case where the ideal $\Ic$ is principal, generated by a monomial.
 $$\Ic=x^{N_1}y^{N_2}$$
 Let $$
 \phi(t)=\left \{
\begin{matrix}
 &x&=& c_1t^{k_1}+\cdots+a_nt^n\\
 &y&=& c_2t^{k_2}+\cdots+b_nt^n
 \end{matrix}
  \right . $$
  where $c_1,c_2\in \bc^*$.
  We have
  $$\text{ord}_\phi \Ic=k_1N_1+k_2N_2=n$$
  $$\text{ord}(\og \circ \phi )=(\nu-1)k_1=m$$
 $$ \mu(V_{n,m})=[\pi_n(V_{n,m})]\bl^{-2n}$$
 Then
 $$\zg_{DL}(\Ic,\og)(T)=(\bl-1)^2\sum_{k_1\geq1,k_2\geq 2}\bl^{n-k_1}\bl^{n-k_2}\bl^{-(\nu-1)k_1}\bl^{-2n}T^n$$
 $$\zg_{DL}(\Ic,\og)(T)=(\bl-1)^2\sum_{k_1\geq1,k_2\geq 2}\bl^{-\nu k_1-k_2}T^{k_1N_1+k_2N_2}$$
 Finally,
 $$\zg_{DL}(\Ic,\og)(T)=(\bl-1)^2\frac{\bl^{-\nu}T^{N_1}}{1-\bl^{-\nu}T^{N_1}}\frac{\bl^{-1}T^{N_2}}{1-\bl^{-1}T^{N_2}}$$

 To compute the Zeta function of an ideal,  we  use induction on the depth of the ideal. If the ideal has depth $0$, we already computed the Zeta function. Consider an ideal $\Ic$ of depth $n$. We give a
 formula computing the Zeta function if $\Ic$ in terms of the Zeta function of the ideals $\Ic_{\sg}$.

 Let $\Ic =(x^N)\Ic'$. Consider its Newton polygon and the dual of its Newton polygon. Let $S$ be a face of the Newton polygon of $\Ic$ with equation $p\ag+q\bg=N$. This face corresponds in the dual to a line going through the origin with direction $(p,q)$. We have a decomposition of $\br_+^2$ into disjoint cones. Let $\Dg$ be one of them generated by $(p_1,q_1),(p_2,q_2)$. Let
 $$\pc_{\Dg}=\{(i,j)\in \bn^2\vert i=\mu_1p_1+\mu_2p_2,j=\mu_1q_1+\mu_2q_2, (\mu_1,\mu_2)\in \bq^2, 0<\mu_1\leq1,0<\mu_2\leq1\}$$

 The set of points with integer coordinates in the cone $\Dg$ is
 $$\rc_{\Dg}=\{(i,j)\in \bn^2\vert i=i_0+k_1p_1+k_2p_2,j=j_0+k_1q_1+b_2q_2,(i_0,j_0)\in \pc_{\Dg},(k_1,k_2)\in \bn^2\}$$

 Let $(k_1,k_2)\in \bn^2$
 \begin{enumerate}
 \item Assume that $(k_1,k_2)$ lies in a cone $\Dg$ generated by $(p_1,q_1),(p_2,q_2)$. Let $(a,b)$ the intersection point of the two faces of the Newton polygon whose supporting lines have equations
$ p_1 \ag+q_1\bg =N_1$ and  $ p_2 \ag+q_2\bg =N_2$.  Let $\phi \in \cl _0(\bc^2)$ $$
 \phi (t)=\left \{
\begin{matrix}
 &x&=& c_1t^{k_1}+\cdots+a_nt^n\\
 &y&=& c_2t^{k_2}+\cdots+b_nt^n
 \end{matrix}
  \right . $$
  where $c_1,c_2\in \bc^*$.
   $$\text{ord}_\phi \Ic=k_1a+k_2b=n$$
  $$\text{ord}(\og \circ \phi )=(\nu-1)k_1$$
  The contribution to $\zg_{DL}(\Ic,\og)(T)$ is
  $$(\bl-1)^2\sum_{(k_1,k_2)\in \Dg\cap \bn^2}\bl^{n-k_1}\bl^{n-k_2}\bl^{-2n}\bl^{-(\nu-1)k_1}T^n$$
   $$=(\bl-1)^2\sum_{(k_1,k_2)\in \Dg\cap \bn^2}\bl^{-\nu k_1}\bl^{-k_2}T^{k_1a+k_2b}$$
   Let
   $$D_{\Dg}=\sum_{(i,j)\in \pc_{\Dg}}\bl^{-(\nu i+j)}T^{ia+jb}$$
   Then the contribution to $\zg_{DL}(\Ic,\og)(T)$ is
$$(\bl-1)^2 D_{\Dg}  \sum_{(n_1,n_2)\in \bn^2} \bl ^{-[(p_1\nu+q_1)n_1+(p_2\nu+q_2)n_2]}T^{(p_1a+q_1b)n_1+(p_2a+q_2b)n_2}$$
$$=(\bl-1)^2 D_{\Dg}  \frac{1}{(1-\bl^{-(p_1\nu+q_1)}T^{N_1})(1-\bl^{-(p_2\nu+q_2)}T^{N_2})}$$
 \item Assume that $(k_1,k_2)\in L$ where $L$ is a line separating two cones, generated by $(p,q)$. Then there exists $k\in \bn^*$ such that $k_1=pk, k_2=qk$. Let $S$ be the corresponding face of the Newton polygon and $F_{l_S}$ be the face polynomial.  Let $\phi \in \cl _0(\bc^2)$ $$
 \phi (t)=\left \{
\begin{matrix}
 &x&=& c_1t^{pk}+\cdots+a_nt^n\\
 &y&=& c_2t^{qk}+\cdots+b_nt^n
 \end{matrix}
  \right . $$
  where $c_1,c_2\in \bc^*$.
We have two cases
\begin{enumerate}
\item For any root $\mu$ of $F_{l_S}$, we have $c_1^q-\mu c_2^p\neq 0$. In this case,
  $$\text{ord}_\phi \Ic=Nk$$
  $$\text{ord}(\og \circ \phi )=p(\nu-1)k$$
  The contribution to $\zg_{DL}(\Ic,\og)(T)$ is
   $$(\bl-1)(\bl-r-1)\sum_{k\in \bn^*}\bl^{-(p\nu+q)k}T^{Nk}$$
   where $r$ is the number of distinct roots of $F_{l_S}$.
   The contribution to $\zg_{DL}(\Ic,\og)(T)$ is
  $$(\bl-1)(\bl-r-1)\frac{\bl ^{-(p\nu+q)}T^N}{1-\bl^{- (p\nu +q)}T^N}$$
\item There exists $\mu$ root of $F_{l_S}$ such that $c_1^q-\mu c_2^p= 0$.

We need a lemma to achieve the computation. Denote by $\overline{X}_{n,(\nu-1)pk}$ the set of arcs such that
$$\overline{ \phi} (t)=\left \{
\begin{matrix}
 &x&=& c_1t^{pk}+\cdots+a_nt^n+a_{n+1}t^{n+1}+\cdots+a_{n+pk}t^{n+pk}\\
 &y&=& c_2t^{qk}+\cdots+b_nt^n+b_{n+1}t^{n+1}+\cdots+b_{n+qk}t^{n+qk}
 \end{matrix}
 \right.$$
 and $\text{ord}_{\overline{\phi}}\Ic =n, \text{ord}( \og \circ \overline{\phi})=(\nu-1)pk$

 Denote by $\overline{X}^{\sg}_{n,(p\nu+q-1)k}$ the set of arcs such that
$$\overline{ \phi}_1 (t)=\left \{
\begin{matrix}
 &x_1&=& c_1t^{k}+\cdots+a_nt^n+a_{n+1}t^{n+1}+\cdots+a_{n+k}t^{n+k}\\
 &y_1&=& c_2t+\cdots+b_nt^n
 \end{matrix}
 \right.$$
 and $\text{ord}_{\overline{\phi}_1}\Ic_{\sg} =n, \text{ord}( \og_1 \circ \overline{\phi}_1)=(p\nu+q-1)k$,
 where $\og_1=x_1^{p\nu+q-1}dx_1\wedge dy_1$.
 Let $\overline{ \phi}\in \overline{X}_{n,(\nu-1)pk}$, and write
$$\overline{ \phi} (t)=\left \{
\begin{matrix}
 &x&=& c_1t^{pk}\phi_1(t)\\
 &y&=& c_2t^{qk}\phi_2(t)
 \end{matrix}
 \right.$$
 Define
 $$\textit{S}(\overline{ \phi}) (t)=\left \{
\begin{matrix}
 &x&=& \frac{c_2^{q'}}{c_1^{p'}}t^{k}\phi_1(t)^{1/p} & \ \ \mod t^{n+k+1}\\
 &y&=&-\mu^{p'}+\mu^{p'} \frac{\phi_2(t)}{\phi_1(t)^q/p} &\ \ \mod t^{n+1}
 \end{matrix}
 \right.$$
 \begin{lemma}
 The map $\textit{S}$ defines a bijection between  $\overline{X}_{n,(\nu-1)pk}$ and $\overline{X}^{\sg}_{n,(p\nu+q-1)k}$.
  \end{lemma}
 \begin{proof}
 First $\textit{S}(\overline{ \phi}) \in \overline{X}^{\sg}_{n,(p\nu+q-1)k}$. Let's consider
 $$\overline{ \phi}_1 (t)=\left \{
\begin{matrix}
 &x_1&=& c_1t^{k}+\cdots+a_nt^n+a_{n+1}t^{n+1}+\cdots+a_{n+k}t^{n+k}\\
 &y_1&=& c_2t+\cdots+b_nt^n
 \end{matrix}
 \right.$$
  such that $\text{ord}_{\overline{\phi}_1}\Ic_{\sg} =n, \text{ord}( \og_1 \circ \overline{\phi}_1)=(p\nu+q-1)k$. Let $x_1=c_1t^k\phi (t)$. Define
  $$\overline{ \phi} (t)=\left \{
\begin{matrix}
 &x&=& c_1^p\mu^{q'}t^{pk}\phi(t)^p &\ \ \mod t^{n+pk+1}\\
 &y&=& c_1^qt^{qk}\phi(t)^q(y_1-\mu^{p'})&\ \ \mod t^{n+qk+1}
 \end{matrix}
 \right.$$
 Then $\overline{ \phi} \in \overline{X}_{n,(\nu-1)pk}$.
 \end{proof}

 Now let
 Denote by $X_{n,(\nu-1)pk}$ the set of arcs such that
$$ \phi (t)=\left \{
\begin{matrix}
 &x&=& c_1t^{pk}+\cdots+a_nt^n\\
 &y&=& c_2t^{qk}+\cdots+b_nt^n
 \end{matrix}
 \right.$$
 and $\text{ord}_{\overline{\phi}}\Ic =n, \text{ord}( \og \circ \overline{\phi})=(\nu-1)pk$

 Denote by $X^{\sg}_{n,(p\nu+q-1)k}$ the set of arcs such that
$$ \phi_1 (t)=\left \{
\begin{matrix}
 &x_1&=& c_1t^{k}+\cdots+a_nt^n\\
 &y_1&=& c_2t+\cdots+b_nt^n
 \end{matrix}
 \right.$$
 and $\text{ord}_{\overline{\phi}_1}\Ic_{\sg} =n, \text{ord}( \og_1 \circ \overline{\phi}_1)=(p\nu+q-1)k$.

 We have
 $$\sum_{n\geq 1}\sum_{k\geq 1}[X_{n,(\nu-1)pk}]\bl^{-(\nu-1)pk}\bl^{-2n}T^n=$$
 $$\sum_{n\geq 1}\sum_{k\geq 1}[\overline{X}_{n,(\nu-1)pk}]\bl^{-qk} \bl^{-\nu pk}\bl^{-2n}T^n=$$
 $$\sum_{n\geq 1}\sum_{k\geq 1}[ \overline{X}^{\sg}_{n,(p\nu+q-1)k}]\bl^{-(q+p\nu)}\bl^{-2n}T^n=$$
  $$\sum_{n\geq 1}\sum_{k\geq 1}[ X^{\sg}_{n,(p\nu+q-1)k}]\bl^{-(p\nu+q-1)k}\bl^{-2n}T^n$$
    using the change of variable formula.

    Finally,
  $$\sum_{n\geq 1}\sum_{k\geq 1}[ X^{\sg}_{n,(p\nu+q-1)k}]\bl^{-(p\nu+q-1)k}\bl^{-2n}T^n=
  \sum_{n\geq1}(\sum_{m}\mu(V_{n,m}^{\sg})\bl^{-m})T^n$$
  where $V_{n,m}^{\sg}$ is the set of arcs $\phi \in \cl_0(\bc^2)$ such that
  $\text{ord}_{\phi}\sg(\Ic)=n$ and $\text{ord} (\og_1\circ \phi)=m$.
  Then
  $$ \sum_{n\geq1}(\sum_{m}\mu(V_{n,m}^{\sg})\bl^{-m})T^n=\zg(\sg(\Ic),\og_1)(T)$$
  As the Zeta-function does not depend on the system of coordinates,
 $$\zg(\sg(\Ic),\og_1)(T)=\zg(\Ic_{\sg},\og_1)(T)$$
 \end{enumerate}
 \end{enumerate}
 We have proven

 \begin{theorem}
   $$\zg(f,\og)(T)=\sum_{\Delta} (\bl-1)^2 D_{\Dg}  \frac{1}{(1-\bl^{-(p_1\nu+q_1)}T^{N_1})(1-\bl^{-(p_2\nu+q_2)}T^{N_2})}+$$
 $$  \sum_L (\bl-1)(\bl-r-1)\frac{\bl ^{-(p\nu+q)}T^N}{1-\bl^{- (p\nu +q)}T^N}+$$
 $$\sum_L\sum_{\mu}\zg(\Ic_{\sg},\og_1)(T)$$
 \end{theorem}

\end{document}